\documentclass[12pt,final]{article}

\usepackage{amsmath}
\usepackage{amssymb}
\usepackage{amsthm}
\usepackage{showkeys}
\usepackage{enumerate}

\newtheorem{theorem}{Theorem}
\newtheorem{proposition}[theorem]{Proposition}
\newtheorem{lemma}[theorem]{Lemma}

\theoremstyle{definition}

\newtheorem{remark}[theorem]{Remark}

\numberwithin{equation}{section}
\numberwithin{theorem}{section}

\def\ep{\varepsilon}
\def\vp{\varphi}
\def\vvp{\varPhi}

\def\ttheta{\varTheta}

\def\R{\mathbb{R}}
\def\cale{\mathcal{E}}
\def\calb{\mathcal{B}}
\def\holh{\calb^{1+\sigma/2,1+\sigma}_T}
\def\holp{\calb^{1+\sigma/2,2+\sigma}_T}

\def\rt{\tilde{\rho}}
\def\ut{\tilde{u}}
\def\ttt{\tilde{\theta}}
\def\pt{\tilde{p}}
\def\zt{\tilde{z}}

\def\rh{\hat{\rho}}
\def\uh{\hat{u}}
\def\tth{\hat{\theta}}
\def\xh{\hat{x}}
\def\th{\hat{t}}
\def\vph{\hat{\vp}}
\def\psih{\hat{\psi}}
\def\chih{\hat{\chi}}
\def\vvph{\hat{\vvp}}

\def\ubar{\bar{u}}
\def\tbar{\bar{\theta}}
\def\fbar{\bar{f}}
\def\gbar{\bar{g}}

\def\mnp{\mathcal{M}^+}
\def\mnz{\mathcal{M}^0}
\def\mnn{\mathcal{M}^-}

\def\dels{\delta}
\def\ub{u_{\text{\rm b}}}
\def\tb{\theta_{\text{\rm b}}}
\def\uub{U_{\text{\rm b}}}
\def\ttb{\ttheta_{\text{\rm b}}}
\def\cv{c_{\text{\rm v}}}
\def\cp{c_{\text{\rm p}}}
\def\pr{P_{\text{\rm r}}}
\def\ubh{\hat{u}_{\text{\rm b}}}
\def\tbh{\hat{\theta}_{\text{\rm b}}}
\def\cvh{\hat{c}_{\text{\rm v}}}

\def\hsm{h^{\text{\rm s}}}
\def\hcm{h^{\text{\rm c}}}
\def\thsm{\tilde{h}^{\text{\rm s}}}
\def\thcm{\tilde{h}^{\text{\rm c}}}

\def\pd{\partial}
\def\tr{\mathop{\rm Tr}}
\def\diag{\mathop{\rm diag}}
\def\trp{\intercal}

\def\lp#1#2{\| #2 \|_{L^{#1}}}
\def\lt#1{\|#1\|}
\def\li#1{\lp{\infty}{#1}}
\def\lpa#1#2#3{\| #3 \|_{L^{#1}_{#2}}}
\def\ltaa#1#2{\lpa{2}{#1}{#2}}
\def\lta#1#2{| #2 |_{#1}}

\def\hs#1#2{\| #2 \|_{H^{#1}}}
\def\hsa#1#2#3{\| #3 \|_{H^{#1}_{#2}}}
\def\ho#1{\hs{1}{#1}}
\def\hoa#1#2{\hsa{1}{#1}{#2}}
\def\ltat#1#2{[ \mspace{1.mu} #2 \mspace{1.mu} ]_{#1}}
\def\hsat#1#2#3{|\mspace{-1mu}[\mspace{1mu} #3 \mspace{1mu}]\mspace{-1mu}|_{{#1},{#2}}}
\def\hoat#1#2{\hsat{1}{#1}{#2}}
\def\holx#1#2{\| #2 \|_{\calb^{#1}}}
\def\holtx#1#2#3{\| #3 \|_{\calb_T^{#1,#2}}}

\begin{document}

\abovedisplayskip=8pt plus 2pt minus 4pt
\belowdisplayskip=\abovedisplayskip


\thispagestyle{plain}

\title{%
\bf\Large%
\vspace{-8mm}
Stationary waves
 to  viscous heat-conductive 
 gases \\
 in  half space:\\
existence, stability and convergence rate%
\footnote{%
The second author's work was supported in part by
Grant-in-Aid for Young Scientists (B) 21740100 of 
the Ministry of Education, Culture, Sports, Science and Technology.
The fourth author's work was supported in part by
JSPS postdoctoral fellowship under P99217.
}
}

\markboth%
{Stationary wave to viscous heat-conductive gases}%
{Stationary wave to viscous heat-conductive gases}

\author{%
{\large\sc Shuichi Kawashima}${}^1$,
{\large\sc Tohru Nakamura}${}^1$,
\\
{\large\sc Shinya Nishibata}${}^2$
{\normalsize and}
{\large\sc Peicheng Zhu}${}^3$
}

\date{%
\normalsize
${}^1$%
Faculty of Mathematics, Kyushu University
\\
Fukuoka 812-8581, Japan
\\ [7pt]
${}^2$%
Department of Mathematical and Computing Sciences,
\\
Tokyo Institute of Technology
\\
Tokyo 152-8552, Japan
\\ [7pt]
${}^3$%
Basque Center for Applied Mathematics (BCAM)
\\
Building 500, Bizkaia Technology Park,
E-48160 Derio, Spain
\\
IKERBASQUE, Basque Foundation for Science
\\
E-48011 Bilbao, Spain
\vspace{-5mm}
}

\maketitle

%

%
\begin{abstract}
The main concern of the present paper is to study large-time
behavior of solutions to an ideal polytropic model of
compressible viscous gases in one-dimensional half space.
We consider an outflow problem, where the gas blows out
through the boundary, 
and  obtain a convergence rate of solutions toward
 a corresponding stationary solution.
Here the existence of the stationary solution is proved under 
a smallness condition on the boundary data
 with the aid of center manifold theory.
We also show  the time asymptotic stability of the stationary solution
under smallness assumptions on the boundary data and
the initial perturbation in the Sobolev space, by employing an energy method.
Moreover, the convergence rate of the solution toward the 
stationary solution is obtained, provided  that the initial perturbation 
belongs to the  weighted Sobolev space.
Precisely,  the convergence rate we obtain coincides
 with the spatial decay rate of the 
initial perturbation.
The proof is mainly based on  {\it a priori} estimates 
of the perturbation from the stationary solution, which
are derived
 by  a time and space weighted energy method.
\end{abstract}

\begin{description}

\item[{\it Keywords:}]
Compressible Navier--Stokes equation;
Eulerian coordinate;
ideal polytropic model;
outflow problem;
boundary layer solution;
weighted energy method.

\item[{\it 2000 Mathematics Subject Classification:}]
35B35;  
35B40;  
76N15.  

\end{description}


\clearpage

\section{Introduction and main result}
\label{intro}

\subsection{Formulation of the problem}
\label{intro-problem}

We study large-time behavior of a solution
to an initial boundary value problem for the 
 compressible Navier--Stokes equations
 over one-dimensional
half space $\R_+ := (0,\infty)$.
An ideal polytropic  model of  compressible viscous fluid
is formulated in the Eulerian coordinates as
\begin{subequations}
\label{nse}
\begin{gather}
\rho_t + (\rho u)_x = 0,
\label{eq1}
\\
(\rho u)_t + (\rho u^2 + p(\rho,\theta))_x = \mu u_{xx},
\label{eq2}
\\
\Bigl\{
\rho \Bigl(
\cv \theta + \frac{u^2}{2}
\Bigr)
\Bigr\}_t
+ \Bigl\{
\rho u \Bigl(
\cv \theta + \frac{u^2}{2}
\Bigr)
+ p(\rho,\theta) u
\Bigr\}_x
=
(\mu u u_x + \kappa \theta_x)_x,
\label{eq3}
\end{gather}
\end{subequations}
where unknown functions are $\rho = \rho(t,x)$, $u = u(t,x)$
and $\theta = \theta(t,x)$ standing for a mass density,
 a fluid velocity and an absolute temperature, respectively.
Due to the Boyle--Charles law, a pressure $p$ is 
explicitly given by a function of the density and the absolute
temperature:
\[
p = p(\rho,\theta) 
:= R \rho \theta,
\]
where $R>0$ is a gas constant.
Positive constants $\cv$, $\mu$ and $\kappa$ mean
a specific heat at constant volume, 
a viscosity coefficient
and a thermal conductivity, respectively.
Due to Mayler's relation for the ideal gas,
the specific heat 
 $\cv$ is expressed by
 the gas constant $R$ 
and
an adiabatic constant $\gamma > 1$ as
\[
\cv 
=
\frac{R}{\gamma - 1}.
\]
We also introduce physical constants
\[
\cp
:=
\gamma \cv
= \frac{\gamma}{\gamma-1} R,
\quad
\pr
:=
\frac{\mu}{\kappa} \cp
=
\frac{\mu}{\kappa} \frac{\gamma}{\gamma-1} R,
\]
which stand for a specific heat at constant pressure
and the Prandtl number, respectively.
The Prandtl number plays an important role in analysis
of a property of a stationary solution.

We put an initial condition
\begin{equation}
(\rho,u,\theta)(0,x)
=
(\rho_0,u_0,\theta_0)(x)
\label{ic}
\end{equation}
and boundary conditions
\begin{equation}
u(t,0) = \ub < 0,
\quad
\theta(t,0) = \tb > 0,
\label{bc}
\end{equation}
where $\ub$ and $\tb$ are constants.
It is assumed that the initial data converges to a
constant as $x$ tends to infinity:
\begin{equation*}
\lim_{x \to \infty}
(\rho_0, u_0, \theta_0)(x)
=
(\rho_+, u_+, \theta_+).
\end{equation*}
Moreover, we assume that the initial density and 
absolute temperature are uniformly positive, that is,
\begin{equation*}
\inf_{x \in \R_+} \rho_0(x) > 0,
\quad
\inf_{x \in \R_+} \theta_0(x) > 0,
\quad
\rho_+ > 0,
\quad
\theta_+ > 0.
\end{equation*}
The boundary condition for $u$ in (\ref{bc}) means that
the fluid blows out from the boundary.
Hence
this problem is called an outflow problem (see \cite{matu-01}).
Due to the outflow boundary condition,
the characteristic of the hyperbolic equation (\ref{eq1}) for 
the density $\rho$
is negative around the boundary
so that two boundary conditions are necessary and
sufficient for the wellposedness of this problem.

In the paper \cite{knz03},
Kawashima, Nishibata and Zhu  considered the outflow problem
for an isentropic model and obtained a necessary and sufficient
condition for the existence of the stationary solution.
Moreover, they proved the asymptotic stability of the stationary
solution under the smallness assumption on the initial
perturbation and the strength of the boundary data.
A convergence rate toward the stationary solution for this model was
obtained by Nakamura, Nishibata and Yuge in \cite{nny07}
under the assumption that the initial perturbation belongs to
the  suitably weighted Sobolev space.
The main concern of the present paper is to
extend these results to the model of heat-conductive viscous gas.
Precisely, we  show the existence
and the asymptotic stability of the stationary solution
as well as the convergence rate for the ideal polytropic
model (\ref{nse}).
Compared to the isentropic model, 
the heat-conductive model is more difficult
to handle.
For example,
since  the model (\ref{nse}) has two parabolic equations,
the equations for
the stationary wave are deduced to a $2 \times 2$ system of
autonomous ordinary differential equations.
However, it becomes a scalar equation in the case of
the isentropic flow.
Therefore, to obtain a condition which guarantees the
existence of the stationary solution
for the heat-conductive model,
we have to examine dynamics around an equilibrium of
the system by using center manifold theory.

\subsection{Dimensionless form}
\label{intro-dim}

For the stability analysis on the equations (\ref{nse}),
it is convenient to reformulate the problem into that in the 
dimensionless form.
For this purpose, we define new variables $\xh$ and $\th$ by
\[
\xh := \frac{x}{L},
\quad
\th := \frac{t}{T},
\]
where $L$ and $T$ are positive constants.
We also employ new unknown functions $(\rh, \uh, \tth)$ defined by
\begin{equation}
\rh (\th, \xh)
:= \frac{1}{\rho_+} \rho(t, x),
\quad
\uh (\th, \xh)
:= \frac{1}{|u_+|} u(t,x),
\quad
\tth (\th, \xh)
:= \frac{1}{\theta_+} \theta(t,x).
\label{dml}
\end{equation}
Here we note that the constant $u_+$ must satisfy
\begin{equation}
u_+ < 0
\label{ngu}
\end{equation}
for the existence of the stationary solution.
Indeed, the stationary solution $(\rt,\ut,\ttt)(x)$ satisfies
\begin{equation}
\rt(x) \ut(x) = \rho_+ u_+,
\label{mast}
\end{equation}
which is obtained by integrating $(\rt \ut)_x = 0$ over $(x,\infty)$.
Substituting $x=0$ in (\ref{mast}),
we get $u_+ = \rt(0) \ub / \rho_+$,
which immediately yields (\ref{ngu}) by using the positivity of the
density and 
the boundary condition 
$\ub < 0$.
Next we define dimensionless physical constants by
\begin{equation}
\hat{\mu} := \frac{\mu}{\rho_+ |u_+|^2},
\quad
\hat{\kappa} := \frac{\kappa \theta_+}{\rho_+ |u_+|^4},
\quad
\cvh := \frac{1}{\gamma (\gamma-1)}
\label{dim-phy}
\end{equation}
and a dimensionless pressure by
\[
\hat{p} = \hat{p} (\rh, \tth) := \frac{1}{\gamma} \rh \tth.
\]
We also introduce Mach number $M_+$ at the spatial
asymptotic state:
\[
M_+ :=
\frac{|u_+|}{c_+},
\]
where $c_+ := \sqrt{R \gamma \theta_+}$ is sound speed.
Using the dimensionless constants (\ref{dim-phy}),
we  represent the Prandtl number $\pr$ as
\[
\pr
=
\frac{\hat{\mu}}{\hat{\kappa}} \frac{1}{M_+^2 (\gamma-1)}.
\]
Substituting (\ref{dml}) in (\ref{nse})
and letting $L = |u_+|$ and $T = 1$,
we have the equations for $(\rh,\uh,\tth)$ in the dimensionless form
as
\begin{subequations}
\label{d-nse}
\begin{gather}
\rho_t + (\rho u)_x = 0,
\label{d-eq1}
\\
(\rho u)_t + \Bigl(
\rho u^2 + \frac{1}{M_+^2} p(\rho,\theta)
\Bigr)_x = \mu u_{xx},
\label{d-eq2}
\\
\Bigl\{
\rho \Bigl(
\frac{1}{M_+^2}
\cv \theta + \frac{u^2}{2}
\Bigr)
\Bigr\}_t
+ \Bigl\{
\rho u \Bigl(
\frac{1}{M_+^2}
\cv \theta + \frac{u^2}{2}
\Bigr)
+ \frac{1}{M_+^2}
p(\rho,\theta) u
\Bigr\}_x
=
(\mu u u_x + \kappa \theta_x)_x.
\label{d-eq3}
\end{gather}
\end{subequations}
In the equations (\ref{d-nse}), without any confusion,
we abbreviate the symbol `` $\hat{}$ '' to
express dimensionless quantities.
The initial and the boundary conditions for the
dimensionless function $(\rho,u,\theta)$ are
prescribed as
\begin{subequations}
\label{d-ic}
\begin{gather}
(\rho,u,\theta)(0,x)
=
(\rh_0, \uh_0, \tth_0)(x)
:=
\Bigl(
\frac{\rho_0}{\rho_+}, \frac{u_0}{|u_+|}, \frac{\theta_0}{\theta_+}
\Bigr)(x),
\\
\lim_{x \to \infty} (\rh_0, \uh_0, \tth_0)(x)
= (1, -1, 1),
\label{d-icbc}
\end{gather}
\end{subequations}%
\begin{equation}
(u, \theta)(t,0)
= (\ubh, \tbh)
:=
\Bigl( \frac{\ub}{|u_+|}, \frac{\tb}{\theta_+} \Bigr).
\label{d-bc}
\end{equation}
We also abbreviate the hat `` $\hat{}$ '' and 
write the dimensionless initial data
and boundary data as $(\rho_0,u_0,\theta_0)$ 
and $(\ub,\tb)$ respectively in (\ref{d-ic}) and (\ref{d-bc}).

\subsection{Main results}
\label{intro-main}

The main concern of the present paper is to consider the
large-time behavior of solutions to the problem (\ref{d-nse}),
(\ref{d-ic}) and (\ref{d-bc}).
Precisely we show that the solution converges to a stationary
solution $(\rt,\ut,\ttt)(x)$, which is a solution to (\ref{d-nse})
independent of time variable $t$.
Thus the stationary solution $(\rt,\ut,\ttt)$ satisfies
the system
\begin{subequations}
\label{ste}
\begin{gather}
(\rt \ut)_x = 0,
\label{st1}
\\
\Bigl(
\rt \ut^2 + \frac{1}{M_+^2} \pt
\Bigr)_x
=
\mu \ut_{xx},
\label{st2}
\\
\Bigl\{
\rt \ut \Bigl(
\frac{1}{M_+^2} \cv \ttt + \frac{\ut^2}{2}
\Bigr)
+ \frac{1}{M_+^2} \pt\ut
\Bigr\}_x
=
(\mu \ut \ut_x + \kappa \ttt_x)_x,
\label{st3}
\end{gather}
\end{subequations}
where $\pt := p(\rt,\ttt)$.
The stationary solution is supposed to satisfy the
same boundary condition (\ref{d-bc}) and the same
spatial asymptotic condition (\ref{d-icbc}):
\begin{equation}
(\ut, \ttt)(0) = (\ub, \tb),
\quad
\lim_{x \to \infty} (\rt, \ut, \ttt)(x) 
= (1, -1, 1).
\label{st-bc}
\end{equation}
We summarize the existence and the decay property of the
stationary solution $(\rt,\ut,\ttt)$ satisfying (\ref{ste})
and (\ref{st-bc}) in the following proposition.
To this end, we define a boundary strength $\dels$ as 
\[
\dels := |(\ub + 1, \tb - 1)|.
\]

\begin{proposition}
\label{pro-ex}
Suppose that the boundary data $(\ub,\tb)$ satisfies
\begin{equation}
(\ub, \tb)
\in
\mnp
:=
\{
(u,\theta) \in \R^2
\; ; \;
|(u+1, \theta-1)| < \ep_0
\}
\label{sm-del}
\end{equation}
for a certain positive constant $\ep_0$.
Notice that 
the condition \eqref{sm-del} is equivalent to
$\dels < \ep_0$.

\vspace{-2mm}
\begin{enumerate}[\hspace{-7pt}\rm (i)]
\setlength{\itemsep}{0pt}

\item
For the supersonic case $M_+ > 1$,
there exists a unique smooth solution $(\rt,\ut,\ttt)$ to the
problem \eqref{ste} and \eqref{st-bc} satisfying
\begin{equation}
|\pd_x^k (\rt(x)-1, \ut(x)+1, \ttt(x)-1)|
\le
C \dels e^{-cx}
\ \; \text{for} \ \;
k = 0,1,2,\dots,
\label{dc-sp}
\end{equation}
where $C$ and $c$ are positive constants.

\item
For the transonic case $M_+ = 1$,
there exists a certain region $\mnz \subset \mnp$ such that
if the boundary data $(\ub,\tb)$ satisfies the condition
\begin{equation}
(\ub,\tb) \in \mnz,
\label{bc-mnz}
\end{equation}
then there exists a unique smooth solution $(\rt,\ut,\ttt)$ 
satisfying
\begin{equation}
|\pd_x^k (\rt(x)-1, \ut(x)+1, \ttt(x)-1)|
\le
C \frac{\dels^{k+1}}{(1+\dels x)^{k+1}}
+ C \dels e^{-cx}
\ \; \text{for} \ \;
k = 0,1,2,\dots.
\label{dc-tr}
\end{equation}

\item
For the subsonic case $M_+ < 1$,
there exists a certain curve $\mnn \subset \mnp$ such that
if the boundary data $(\ub,\tb)$ satisfies the condition
\begin{equation}
(\ub,\tb) \in \mnn,
\label{bc-mnn}
\end{equation}
then there exists a unique smooth solution $(\rt,\ut,\ttt)$ 
satisfying \eqref{dc-sp}.

\end{enumerate}
\end{proposition}

\begin{figure}[tb]
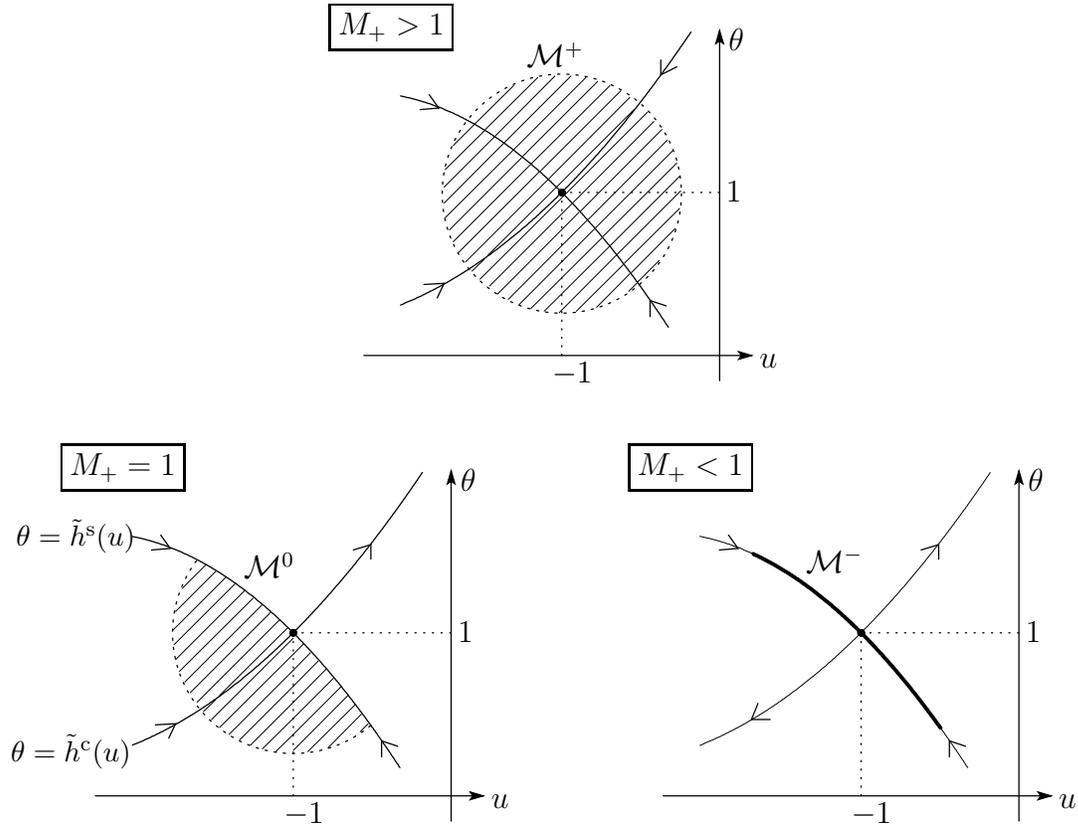

\begin{center}
\input{fig01a.tex}
\hspace{15mm} \ 

\vspace{10mm}

\hspace*{-15mm}
\input{fig02a.tex}
\hspace{0mm}
\input{fig03a.tex}
\end{center}
\caption{%
For the transonic case $M_+=1$, 
the region $\mnz$ consists of one side of $\mnp$ divided by
the local stable manifold $\theta = \thsm(u)$.
For the subsonic case $M_+ < 1$,
the curve $\mnn$ coincides with the local stable manifold.
}
\label{fig-mani}
\end{figure}

The rough sketches of the regions $\mnp$, $\mnz$ and $\mnn$ are
drawn in Figure \ref{fig-mani}.
The precise definitions of $\mnz$ and $\mnn$ are given
in (\ref{def-mn}).
The boundary of $\mnz$, which is the stable manifold for the
stationary problem,
is a curve in the state space.
The geometric property of this curve is completely characterized
by the Prandtl number.
This observation is discussed in Section \ref{st-mani}.

The asymptotic stability of the stationary solution $(\rt,\ut,\ttt)$
is stated in the next theorem.

\begin{theorem}
\label{th-sb}
Suppose that the stationary solution $(\rt,\ut,\ttt)$ exists.
Namely it is assumed that one of the following three conditions holds:
{\rm (i)} $M_+ > 1$ and \eqref{sm-del}, \ 
{\rm (ii)} $M_+ = 1$ and \eqref{bc-mnz}, \ 
{\rm (iii)} $M_+ < 1$ and \eqref{bc-mnn}.
%
%
%
%
%
%
%
In addition, the initial data $(\rho_0,u_0,\theta_0)$ is supposed
to satisfy
\begin{gather*}
\rho_0 \in \calb^{1+\sigma} (\R_+),
\quad
(u_0,\theta_0)
\in
\calb^{2+\sigma} (\R_+),
\\
(\rho_0,u_0,\theta_0) - (\rt,\ut,\ttt)
\in
H^1 (\R_+)
\end{gather*}
for a certain constant $\sigma \in (0,1)$.
Then there exists a positive constant $\ep_1$ such that
if
\[
\ho{(\rho_0,u_0,\theta_0) - (\rt,\ut,\ttt)} + \dels \le \ep_1,
\] 
then the initial boundary value problem \eqref{d-nse}, \eqref{d-ic}
and \eqref{d-bc} has a unique solution globally in time
satisfying
\begin{equation}
\begin{gathered}
\rho \in \holh,
\quad
(u, \theta) \in \holp,
\\
(\rho, u, \theta) - (\rt, \ut, \ttt)
\in
C ([0,\infty) ; H^1(\R_+))
\end{gathered}
\label{global-sp}
\end{equation}
for an arbitrary $T > 0$.
Moreover, the solution $(\rho,u,\theta)$ converges to the
stationary solution $(\rt,\ut,\ttt)$ uniformly 
as time tends to infinity:
\begin{equation}
\lim_{t \to \infty}
\li{(\rho, u, \theta)(t) - (\rt,\ut,\ttt)} 
= 0.
\label{conv-1}
\end{equation}
\end{theorem}

We also show a convergence rate for the stability \eqref{conv-1}
by assuming additionally that the initial perturbation 
belongs to the weighted Sobolev space.

\begin{theorem}
\label{th-cv}
Suppose that the same conditions as in Theorem {\rm \ref{th-sb}}
hold.

\vspace{-2mm}
\begin{enumerate}[\hspace{-7pt}\rm (i)]
\setlength{\itemsep}{0pt}

\item
For the supersonic case $M_+>1$,
if the initial perturbation satisfies
\begin{equation*}
(\rho_0,u_0,\theta_0) - (\rt,\ut,\ttt)
\in
L^2_\alpha(\R_+)
\end{equation*}
for a certain positive constant $\alpha$,
then the solution $(\rho,u,\theta)$ to \eqref{d-nse}, \eqref{d-ic}
and \eqref{d-bc} satisfies the decay estimate
\begin{equation}
\li{(\rho,u,\theta)(t) - (\rt,\ut,\ttt)}
\le
C (1+t)^{-\alpha/2}.
\label{conv-sp}
\end{equation}

\item
For the transonic case $M_+=1$,
let $\alpha \in [1,2(1+\sqrt{2}))$.
There exists a positive constant $\ep_2$ such that if
\begin{equation*}
\dels^{-1/2}
\hoa{\alpha}{(\rho_0,u_0,\theta_0) - (\rt,\ut,\ttt)}
\le
\ep_2,
\end{equation*}
then the solution $(\rho,u,\theta)$ satisfies
the decay estimate
\begin{equation}
\li{(\rho,u,\theta)(t) - (\rt,\ut,\ttt)}
\le
C (1+t)^{-\alpha/4}.
\label{conv-tr}
\end{equation}
\end{enumerate}
\end{theorem}

\begin{remark}
(i) 
For the supersonic case $M_+ > 1$, we can prove an exponential
convergence rate
\begin{equation*}
\li{(\rho,u,\theta)(t) - (\rt,\ut,\ttt)}
\le
C e^{-\alpha t}
\end{equation*}
provided that the initial data satisfies the conditions as in
Theorem \ref{th-sb} and 
$$ (\rho_0,u_0,\theta_0) - (\rt,\ut,\ttt) \in 
 L^2_{\zeta,\text{\rm exp}} (\R_+)
:= \{ u \in L^2_{\text{\rm loc}} (\R_+) ; e^{(\zeta/2) x} u \in
     L^2 (\R_+) \},$$
where $\alpha$ is a positive constant depending on $\zeta$.
Since the proof is almost same as that for the isentropic
model studied in the paper \cite{nny07}, we omit the details.
\\
(ii) 
To obtain the convergence rates (\ref{conv-sp}) and (\ref{conv-tr}),
we derive weighted energy estimates.
In the derivation,
we essentially use a property that all of characteristics
of a hyperbolic system, 
which is obtained by letting $\mu = 0$ and $\kappa = 0$ in (\ref{d-nse}),
are non-positive at spatial asymptotic state.
However, for the subsonic case $M_+ < 1$, one characteristic is
positive.
Due to this, it is difficult to obtain a convergence rate
for the subsonic case by using the weighted energy method.
\\
(iii) 
Compared with the results in \cite{km-85,mn-94,n85-bg} considering
the convergence rate for a scalar viscous conservation law,
the convergence rates in (\ref{conv-sp}) and (\ref{conv-tr}) seem optimal.
For the transonic case, owing to the degenerate property of
the stationary solution, the weight exponent $\alpha$ needs
to be less than a certain constant, i.e., $\alpha < 2(1+\sqrt{2})$.
This kind of restriction on the weight exponent is also necessary 
to obtain a convergence rate $O(t^{-\alpha/4})$ toward the 
degenerate nonlinear waves
for a scalar viscous conservation law and an isentropic model
studied in the papers \cite{mn-94,nny07,unk-ns-pre,unk08}.
We note that, in the papers \cite{unk-ns-pre,unk08},
 the same restriction $\alpha < 2(1+\sqrt{2})$ is also required for 
an isentropic model and 
a scalar viscous conservation law $u_t + f(u)_x = u_{xx}$, 
where a degeneracy exponent is equal to 1, that is,
$f(u) = C (u-u_+)^2 + O(|u-u_+|^3)$.
Recently, Kawashima and Kurata in \cite{kur08}
studied the stability of the degenerate stationary solution for a
viscous conservation law and
obtained the same convergence rate $O(t^{-\alpha/4})$
by using the weighted energy method combined with
the Hardy type inequality
 under a more moderate restriction $\alpha <5$,
which is best possible in the sense that the linearized operator
around the degenerate stationary solution is not  dissipative
in $L^2_\alpha$ for $\alpha > 5$.
\end{remark}

\noindent {\bf Related results.}
From the pioneering work \cite{oleinik} by Il'in and Ole{\u\i}nik,
there have been many studies on the stability of several
nonlinear waves for a scalar viscous conservation law.
For instance,
Kawashima, Matsumura and Nishihara in \cite{km-85,mn-94,n85-bg}
obtained a convergence rate toward a traveling wave
for the Cauchy problem.
For a one-dimensional half space problem,
Liu, Matsumura and Nishihara in \cite{lmn98} considered
the stability of the stationary solution.

For the half space problem of the isentropic model,
Kawashima, Nishibata and Zhu \cite{knz03} proved the existence and the
asymptotic stability of the stationary solution for
the outflow problem.
The convergence rate for this stability result was obtained 
by Nakamura, Nishibata and Yuge in \cite{nny07} by assuming
that the initial perturbation decays in a spatial direction.
The generalization of this one-dimensional outflow problem
to the multi-dimensional half space problem were 
studied by Kagei, Kawashima, Nakamura and Nishibata in \cite{kg06,nn07-p}.
Precisely, Kagei and Kawashima in \cite{kg06} proved the
asymptotic stability of a planar stationary solution in a suitable 
Sobolev space.
The convergence rate was obtained by Nakamura and Nishibata 
in \cite{nn07-p}.
There are also several works on the stationary problem for the
Boltzmann equation (or BGK model) in half space.
See \cite{aoki-91,aoki-90} for numerical computations and
\cite{sone-98} for asymptotic analysis.

\bigskip

\noindent {\bf Outline of the paper.}
The remainder of the present paper is organized as follows.
In Section \ref{st}, we discuss the existence of the
stationary solution and present the proof of Proposition \ref{pro-ex}.
In Section \ref{st-deg}, we show a precise decay property
of the degenerate stationary solution, which is utilized
in the stability analysis   of the degenerate stationary
solution.
In Section \ref{est},  Theorem \ref{th-sb} is proved 
by deriving uniform {\it a priori} estimates of the perturbation
from the stationary solution in $H^1$ Sobolev space
by  an energy method.
Finally, in Section \ref{w-est}, we prove Theorem \ref{th-cv}.
The crucial argument is to derive time and space weighted energy
estimates.
For the supersonic case, in Section \ref{w-est-nd},
we obtain the weighted estimates in $L^2$
space and combine it with the uniform estimates in $H^1$
obtained in Section \ref{est}.
Then we obtain the convergence rate (\ref{conv-sp})
with the aid of induction.
However, owing to the degenerate property of the transonic flow,
we have to derive the weighted estimate not only in $L^2$
but also in $H^1$ in order to obtain the convergence rate (\ref{conv-tr}).
This is discussed in Section \ref{w-est-d}.

\bigskip

\noindent {\bf Notations.}
The Gaussian bracket $[x]$ denotes 
the greatest integer which does not exceed $x$.
For $p \in [1,\infty]$, $L^p (\R_+)$ denotes the
standard Lebesgue space over $\R_+$ 
equipped with the norm $\lp{p}{\cdot}$.
We use the notation $\lt{\cdot} := \lp{2}{\cdot}$.
For a non-negative integer $s$, $H^s (\R_+)$ 
denotes the $s$-th order Sobolev space over $\R_+$
in the $L^2$ sense with the norm
\[
\hs{s}{u}
:=
\Bigl(
\sum_{k=0}^s \lt{\pd_x^k u}^2
\Bigr)^{1/2}.
\]

For constants $p \in [1,\infty)$ and $\alpha \in \R$,
$L^p_\alpha (\R_+)$ denotes the algebraically weighted
$L^p$ space defined by
$L^p_\alpha (\R_+) := 
\{ u \in L^p_{\text{loc}} (\R_+) 
\; ; \; \lpa{p}{\alpha}{u} < \infty \}$
equipped with the norm
\[
\|u \|_{L^p_\alpha}
:=
\Bigl(
\int_{\R_+} (1+x)^\alpha |u(x)|^p \, dx
\Bigr)^{1/p}.
\]
We also use the notation $\lta{\alpha}{\cdot} := \ltaa{\alpha}{\cdot}$.
The space $H^s_\alpha (\R_+)$ denotes the algebraically weighted
$H^s$ space corresponding to $L^2_\alpha (\R_+)$ defined by
$H^s_\alpha (\R_+) := \{ u \in L^2_\alpha(\R_+) \; ; \; 
 \pd_x^k u \in L^2_\alpha (\R_+) \ \text{for} \ k = 0,\dots,s \}$,
equipped with the norm
\[
\hsa{s}{\alpha}{u}
:=
\Bigl(
\sum_{k=0}^s \lta{\alpha}{\pd_x^k u}^2
\Bigr)^{1/2}.
\]

For $\alpha \in (0,1)$, $\calb^{\alpha} (\R_+)$ denotes
the space of 
the H\"older continuous functions over $\R_+$
with the H\"older exponent $\alpha$ with respect to $x$.
For a non-negative integer $k$,
$\calb^{k+\alpha} (\R_+)$ denotes the
space of functions satisfying $\pd_x^i u \in \calb^{\alpha} (\R_+)$
for an arbitrary $i = 0,\dots,k$
equipped with the norm $\holx{k+\alpha}{\cdot}$.
For $\alpha, \beta \in (0,1)$ and $T > 0$,
$\calb^{\alpha,\beta} ([0,T] \times \R_+)$ denotes 
the space of
the H\"older continuous functions over $[0,T] \times \R_+$
with the H\"older exponents $\alpha$ and $\beta$
with respect to $t$ and $x$, respectively.
For non-negative integers $k$ and $\ell$,
$\calb^{k+\alpha,\ell+\beta}_T
:=
\calb^{k+\alpha,\ell+\beta} ([0,T] \times \R_+)$ denotes the
space of functions satisfying 
$\pd_t^i u$, $\pd_x^j u \in \calb^{\alpha,\beta} ([0,T] \times \R_+)$
for arbitrary $i = 0,\dots,k$ and $j = 0,\dots,\ell$
equipped with the norm $\holtx{k+\alpha}{\ell+\beta}{\cdot}$.

\section{Existence of stationary solution}
\label{st}

This section is devoted to showing Proposition \ref{pro-ex}.
Precisely we prove the existence of a solution to the
stationary problem (\ref{ste}) and (\ref{st-bc}).
To this end, we reformulate the  problem (\ref{ste})
and (\ref{st-bc}) into a $2 \times 2$ autonomous system of 
ordinary differential equations of first order.

\subsection{Reformulation of stationary problem}
\label{st-prob}

Integrating \eqref{st1} over $(x,\infty)$, we have
\begin{equation}
\rt(x) \ut(x) = -1.
\label{ss01}
\end{equation}
Integrating (\ref{st2}) and (\ref{st3}) over $(x,\infty)$ and
substituting (\ref{ss01}) in the resultant,
we obtain the system of equations for
$
(\ubar, \tbar) (x) := (\ut, \ttt)(x) - (-1, 1)
$
as
\begin{equation}
\frac{d}{dx}
\begin{pmatrix}
\ubar \\ \tbar
\end{pmatrix}
=
J
\begin{pmatrix}
\ubar \\ \tbar
\end{pmatrix}
+
\begin{pmatrix}
\fbar(\ubar, \tbar) \\
\gbar(\ubar, \tbar)
\end{pmatrix},
\label{ode-eq}
\end{equation}
where $J$ is the Jacobian matrix at an equilibrium point $(0,0)$ defined by
\begin{equation*}
J :=
\begin{pmatrix}
\frac{1}{\mu} (\frac{1}{M_+^2 \gamma} - 1)
&
\frac{1}{\mu M_+^2 \gamma}
\\
\frac{1}{\kappa M_+^2 \gamma}
&
- \frac{\cv}{\kappa M_+^2}
\end{pmatrix},
\end{equation*}
and $\fbar$ and $\gbar$ are nonlinear terms defined by
\begin{gather}
\quad
\fbar (\ubar, \tbar)
:=
- \frac{\ubar (\ubar + \tbar)}{\mu M_+^2 \gamma (\ubar - 1)},
\quad
\gbar (\ubar, \tbar)
:=
\frac{\ubar^2}{2 \kappa}.
\nonumber
\end{gather}
Boundary conditions for $(\ubar,\tbar)$ are derived from
(\ref{st-bc}) as
\begin{equation}
(\ubar,\tbar)(0)
= (\ub + 1, \tb - 1),
\quad
\lim_{x \to \infty}
(\ubar, \tbar)(x)
= (0,0).
\label{ode-bc}
\end{equation}
To prove the existence of the stationary solution $(\rt,\ut,\ttt)$,
it suffices to show the existence of the solution $(\ubar,\tbar)$ to
the boundary value problem (\ref{ode-eq}) and (\ref{ode-bc}).
To this end, we diagonalize the system (\ref{ode-eq}).
Let $\lambda_1$ and $\lambda_2$ be eigenvalues of the 
Jacobian matrix $J$.
Since we see later that $J$ has real eigenvalues,
we assume $\lambda_1 \ge \lambda_2$.
Let $r_1$ and $r_2$ be eigenvectors of $J$ corresponding to
$\lambda_1$ and $\lambda_2$, respectively, and let
 $P := (r_1, r_2)$ be a matrix.
Furthermore, using the matrix $P$,
we employ new unknown functions $U(x)$ and $\ttheta(x)$
defined by
\begin{equation}
\begin{pmatrix}
U(x) \\ \ttheta(x)
\end{pmatrix}
:=
P^{-1} 
\begin{pmatrix}
\ubar(x) \\ \tbar(x)
\end{pmatrix}.
\label{tr-ut}
\end{equation}
We also define a corresponding boundary data and nonlinear terms by
\[
\begin{pmatrix}
\uub \\ \ttb
\end{pmatrix}
:=
P^{-1}
\begin{pmatrix}
\ub +1 
\\
\tb -1
\end{pmatrix},
\quad
\begin{pmatrix}
f(U,\ttheta)
\\
g(U,\ttheta)
\end{pmatrix}
:=
P^{-1}
\begin{pmatrix}
\fbar(\ubar,\tbar)
\\
\gbar(\ubar,\tbar)
\end{pmatrix}.
\]
Using these notations,
we rewrite the problem (\ref{ode-eq}) and (\ref{ode-bc}) 
in a diagonal form as
\begin{gather}
\frac{d}{dx}
\begin{pmatrix}
U \\ \ttheta
\end{pmatrix}
=
\begin{pmatrix}
\lambda_1 & 0
\\
0 & \lambda_2
\end{pmatrix}
\begin{pmatrix}
U \\ \ttheta
\end{pmatrix}
+
\begin{pmatrix}
f(U,\ttheta)
\\
g(U,\ttheta)
\end{pmatrix},
\label{d-ode-eq}
\\
(U,\ttheta)(0)
=
(\uub, \ttb),
\quad
\lim_{x \to \infty}
(U, \ttheta)(x)
= 
(0,0).
\label{d-ode-bc}
\end{gather}
Since the existence of the solution to the problem (\ref{ste})
and (\ref{st-bc}) follows from that to the problem (\ref{d-ode-eq})
and (\ref{d-ode-bc}),
here we  show the latter.
Firstly, we consider the case $M_+ > 1$.
Since a discriminant of an eigen-equation of the matrix $J$ satisfies
\[
(\tr J)^2 - 4 \det J
=
(b-c)^2 + a^2 + 2ab + 2ca > 0,
\]
where $a$, $b$ and $c$ are constants defined by
\[
a
:=
\frac{\gamma-1}{\mu M_+^2 \gamma},
\quad
b 
:=
\frac{M_+^2-1}{\mu M_+^2},
\quad
c
:=
\frac{\cv}{\kappa M_+^2},
\]
the eigenvalues $\lambda_1$ and $\lambda_2$ are real numbers.
Moreover we see
\[
\lambda_1 + \lambda_2
= \tr J
= 
-(a+b+c)
< 0,
\quad
\lambda_1 \lambda_2
= \det J
= bc > 0,
\]
which show that $\lambda_1 < 0$ and $\lambda_2 < 0$.
Thus, the equilibrium point $(0,0)$ of (\ref{d-ode-eq}) is
asymptotically stable.
Consequently, if $|(\uub,\ttb)|$ is sufficiently small,
the problem (\ref{d-ode-eq}) and (\ref{d-ode-bc}) has a unique
smooth solution $(U,\ttheta)$ satisfying
\begin{equation}
|\pd_x^k (U(x), \ttheta(x))|
\le
C \dels e^{-cx}
\ \; \text{for} \; \
k = 0,1,\dots.
\label{exp-dc}
\end{equation}

Next we study the case $M_+ = 1$.
Since the matrix $J$ satisfies
\[
\tr J = - \frac{\cv d}{\mu \kappa} < 0,
\quad
\det J = 0,
\quad
d
:=
\mu + \kappa (\gamma-1)^2,
\]
the eigenvalues of $J$ are $\lambda_1 = 0$ and 
$\lambda_2 = - \cv d / (\mu \kappa)$
of which
 eigenvectors are explicitly given by
\[
r_1 =
\begin{pmatrix}
-1 \\ 1 - \gamma
\end{pmatrix},
\quad
r_2 =
\begin{pmatrix}
\kappa(1-\gamma)
\\
\mu
\end{pmatrix},
\]
respectively.
Notice that the matrix $P = (r_1,r_2)$ satisfies 
$\det P = -d < 0$.
Thus there exist a local center manifold $\ttheta = \hcm (U)$
and a local stable manifold  $U = \hsm (\ttheta)$
corresponding to the eigenvalues $\lambda_1=0$ and 
$\lambda_2 = - \cv d / (\mu \kappa)$,
respectively.
In order to show the existence of the solution,
we have to examine dynamics on the center manifold.
To this end, 
we employ a solution $\zt = \zt(x)$ to (\ref{d-ode-eq}) restricted on
the center manifold satisfying the equation
\begin{equation}
\zt_x = f(\zt, \hcm(\zt)).
\label{eq-z}
\end{equation}
By virtue of the center manifold theory  in \cite{carr},
there exists a solution $\zt$ to  (\ref{eq-z}) 
such that the solution $(U,\ttheta)$ to (\ref{d-ode-eq}) 
and (\ref{d-ode-bc}) is given by
\begin{gather}
U(x) = \zt(x) + O(\dels e^{-cx}),
\label{uz}
\\
\ttheta(x) = \hcm(\zt(x)) + O(\dels e^{-cx}).
\label{tz}
\end{gather}
Therefore, to obtain the solution $(U,\ttheta)$ to (\ref{d-ode-eq})
and (\ref{d-ode-bc}),
it suffices to show the existence of the solution to (\ref{eq-z})
satisfying $\zt(x) \to 0$ as $x \to \infty$.
We see that the nonlinear terms $f$ and $g$ satisfy
\begin{gather}
f(U,\ttheta)
=
- \frac{\gamma+1}{2 d} U^2
+ 
O \left(
|U|^3 + |U \ttheta| + |\ttheta|^2
\right),
\label{fg1}
\\
g(U, \ttheta)
=
\frac{\gamma-1}{2 \mu d} (\pr - 2)
U^2
+ 
O \left(
|U|^3 + |U \ttheta| + |\ttheta|^2
\right).
\label{fg2}
\end{gather}
Substituting (\ref{fg1}) in (\ref{eq-z}),
we deduce (\ref{eq-z}) to
\begin{equation}
\zt_x
=
- \frac{\gamma+1}{2 d} \zt^2 + O(|\zt|^3),
\label{eq-z2}
\end{equation}
which yields that $\zt$ is monotonically decreasing
for sufficiently small $\zt$.
Thus, to satisfy $\zt(x) \to 0$ as $x \to \infty$,
the boundary data $\zt(0)$ should be positive.
Namely, for the existence of the solution $(U,\ttheta)$,
the boundary data 
$(\uub,\ttb)$ should be located in the right region
from the local stable manifold, that is,
 $(\uub,\ttb)$ should satisfy a condition
\begin{equation}
\uub \ge \hsm(\ttb).
\label{sm-tr}
\end{equation}
From (\ref{eq-z2}), we also see that the solution $\zt$ satisfies
\begin{equation}
0 < c \frac{\dels}{1 + \dels x}
\le
\zt(x)
\le
C \frac{\dels}{1 + \dels x},
\quad
|\pd_x^k \zt(x)|
\le
C \frac{\dels^{k+1}}{(1+\dels x)^{k+1}}.
\label{zt-1}
\end{equation}
Combining (\ref{uz}), (\ref{tz}) and (\ref{zt-1})
with using $\hcm (\zt) = O(\zt^2)$,
we have the decay property of $(U,\ttheta)$:
\begin{equation}
|\pd_x^k (U(x),\ttheta(x))|
\le
C \frac{\dels^{k+1}}{(1+ \dels x)^{k+1}}
+ C \dels e^{-cx}
\ \; \text{for} \; \
k = 0,1,\dots.
\label{pol-dc}
\end{equation}

Finally we prove the existence of the solution to (\ref{d-ode-eq})
and (\ref{d-ode-bc}) for the subsonic case $M_+ < 1$.
For this case,
the eigenvalues of the matrix $J$ are $\lambda_1 > 0$ and
$\lambda_2 < 0$,
 so that there exist a local unstable manifold
and a local stable manifold.
Therefore, the problem (\ref{d-ode-eq}) and (\ref{d-ode-bc}) has
a solution $(U,\ttheta)$ satisfying (\ref{exp-dc})
if the boundary data is located on the stable
manifold, that is,
\begin{equation}
\uub = \hsm (\ttb).
\label{sm-sub}
\end{equation}

We summarize the above observation in Lemma \ref{ex-dode}
as the existence result to the problem (\ref{d-ode-eq})
and (\ref{d-ode-bc}).

\begin{lemma}
\label{ex-dode}
Suppose that $|(\uub,\ttb)|$ is sufficiently small.

\vspace{-2mm}
\begin{enumerate}[\hspace{-7pt}\rm (i)]
\setlength{\itemsep}{0pt}

\item
For the supersonic case $M_+ > 1$,
there exists a unique smooth solution $(U,\ttheta)$ to the problem
\eqref{d-ode-eq} and \eqref{d-ode-bc} satisfying \eqref{exp-dc}.

\item
For the transonic case $M_+ = 1$,
 if the boundary data $(\uub,\ttb)$ satisfies \eqref{sm-tr},
there exists a unique smooth solution $(U,\ttheta)$ satisfying
\eqref{pol-dc}.

\item
For the subsonic case $M_+ < 1$,
 if the boundary data $(\uub,\ttb)$ satisfies \eqref{sm-sub},
there exists a unique smooth solution $(U,\ttheta)$ satisfying
\eqref{exp-dc}.

\end{enumerate}
\end{lemma}

The proof of Proposition \ref{pro-ex} 
follows immediately form Lemma \ref{ex-dode}.
Indeed, by using the conditions (\ref{sm-tr}) and (\ref{sm-sub}),
we precisely define the regions $\mnz$ and $\mnn$ 
 in Proposition \ref{pro-ex} as follows.
Define $\hat{U}(u,\theta)$ and $\hat{\ttheta}(u,\theta)$ by
\begin{equation}
\begin{pmatrix}
\hat{U}(u,\theta) 
\\ 
\hat{\ttheta} (u,\theta)
\end{pmatrix}
:=
P^{-1} 
\begin{pmatrix}
u+1 
\\
\theta-1
\end{pmatrix}.
\label{sm-lg}
\end{equation}
Note that $U(x) = \hat{U}(\ut(x),\ttt(x))$ and 
$\ttheta(x) = \hat{\ttheta} (\ut(x), \ttt(x))$ hold
from (\ref{tr-ut}).
Then,  defining the regions $\mnz$ and $\mnn$ by
\begin{equation}
\begin{aligned}
\mnz
 & :=
\bigl\{
(u,\theta) \in \mnp
\; ; \;
\hat{U} (u,\theta) \ge \hsm \bigl( \hat{\ttheta} (u,\theta) \bigr)
\bigr\},
\\
\mnn
 & :=
\bigl\{
(u,\theta) \in \mnp
\; ; \;
\hat{U} (u,\theta) = \hsm \bigl( \hat{\ttheta} (u,\theta) \bigr)
\bigr\},
\end{aligned}
\label{def-mn}
\end{equation}
we see that 
the conditions (\ref{sm-tr}) and (\ref{sm-sub}) are
equivalent to (\ref{bc-mnz}) and (\ref{bc-mnn}), respectively.

\subsection{Estimates for degenerate stationary solution}
\label{st-deg}

The aim of the present section is to obtain
more delicate estimates of the degenerate stationary
solution,
which will be utilized in deriving {\it a priori} estimates of
 the perturbation from the degenerate stationary solution
for the case $M_+ = 1$.

\begin{lemma}
Suppose that the degenerate stationary solution exists.
Namely, the same conditions as in Proposition {\rm \ref{pro-ex} - (ii)}
are supposed to hold.
Then the degenerate stationary solution $(\rt,\ut,\ttt)$ satisfies
\begin{gather}
(\rt,\ut,\ttt)
=
(1,-1,1)
+ (-1,-1,1-\gamma) \zt
+ O(\zt^2 + \dels e^{-cx}),
\label{dst-1}
\\
(\ut_x,\ttt_x)
=
\frac{\gamma+1}{2 d} (1, \gamma-1) \zt^2
+ O(\zt^3 + \dels e^{-cx}),
\label{dst-2}
\\
|\pd_x^k (\ut,\ttt)|
\le
C \zt^{k+1}
+ C \dels e^{-cx}
\; \ \text{for} \ \;
k = 1,2,\dots.
\label{dst-3}
\end{gather}
\end{lemma}

\begin{proof}
The estimates for $(\ut,\ttt)$ in (\ref{dst-1}) are obtained
by using (\ref{uz}), (\ref{tz}) and
\begin{equation}
\begin{pmatrix}
\ut
\\
\ttt
\end{pmatrix}
=
\begin{pmatrix}
-1
\\
1
\end{pmatrix}
+
P
\begin{pmatrix}
U
\\
\ttheta
\end{pmatrix},
\quad
P=
\begin{pmatrix}
-1 & \kappa (1 - \gamma)
\\
1 - \gamma & \mu
\end{pmatrix}
\label{utou}
\end{equation}
which follows from (\ref{tr-ut}).
Due to the fact that $\rt \ut = -1$,
we have the estimate for $\rt$ in (\ref{dst-1}).
By using (\ref{uz}), (\ref{tz}) and (\ref{eq-z2}),
we see that
\begin{equation}
U_x 
= -\frac{\gamma+1}{2d} \zt^2
+ O(\zt^3 + \dels e^{-cx}),
\quad
\ttheta_x
= O(\zt^3 + \dels e^{-cx}).
\label{uxtx}
\end{equation}
Differentiating (\ref{utou}) in $x$ and substituting (\ref{uxtx})
yield the desired estimate (\ref{dst-2}).
We also have the estimates
 $|\pd_x^k (U,\ttheta)| = O(\zt^{k+1} + \dels e^{-cx})$
inductively, which give the estimate (\ref{dst-3}) due to (\ref{utou}).
Therefore we complete the proof.
\end{proof}

\subsection{Local structure of invariant manifolds}
\label{st-mani}

In order to verify the conditions (\ref{sm-tr}) and (\ref{sm-sub}),
which ensure the existence of the stationary solution, 
it is important to make clear the local shapes of 
the invariant manifolds
$\hcm$ and $\hsm$.
In the present section, we 
  focus ourselves on the transonic case $M_+ = 1$ and 
show that the geometric properties of the invariant manifolds are
characterized by the Prandtl number.
In  detailed arguments, we follow an idea
 in \cite{carr}.
Precisely, we approximate $\hcm$ and $\hsm$ by 
polynomial functions around the equilibrium point as
\begin{equation}
\begin{aligned}
\hcm (U) 
& = 
c_2 U^2 + c_3 U^3 + O(U^4),
\\
\hsm (\ttheta)
& =
s_2 \ttheta^2 + s_3 \ttheta^3 + O(\ttheta^4)
\end{aligned}
\label{loc-mani}
\end{equation}
and obtain precise expressions of the constants $c_i$ and $s_i$
$(i=2,3)$.

Firstly we treat the center manifold $\hcm$.
Differentiating the relation $\ttheta = \hcm (U)$ in $x$,
we have
\begin{equation}
\ttheta_x = (\hcm)'(U) U_x.
\label{mn01}
\end{equation}
Substituting the equation (\ref{d-ode-eq}) in (\ref{mn01}) and
using the relation $\ttheta = \hcm(U)$ again,
we have
\begin{equation}
\lambda_2 \hcm (U) + g(U, \hcm(U))
=
(\hcm)'(U) f(U,\hcm(U)),
\label{mn02}
\end{equation}
where we have used $\lambda_1 = 0$.
Substituting $\lambda_2 = - \cv d / (\mu \kappa)$ and (\ref{fg2})
in (\ref{mn02})
and using the equalities  $$\ttheta = \hcm(U) = O(U^2),\quad
(\hcm)'(U) = O(|U|),\quad  f(U,\hcm(U)) = O(U^2),$$
we get the second order approximation of $\hcm$:
\[
\hcm (U)
=
- \frac{1}{\lambda_2} g(U, \hcm(U)) + O(|U|^3)
=
\frac{\gamma (\gamma-1)^2 \kappa}{2 d^2} (\pr -2) U^2
+ O(|U|^3).
\]
This 
approximation
 means  $c_2$ is given by
\[
c_2 =
\frac{\gamma (\gamma-1)^2 \kappa}{2 d^2} (\pr -2).
\]
For the case of $\pr = 2$, that is, $c_2 = 0$,
we compute $c_3$  similarly as above and get 
\[
c_3 = \frac{\gamma (\gamma-1)^2 \kappa}{d^2} > 0.
\]

Next we obtain $s_2$ and $s_3$.
Differentiating $U = \hsm(\ttheta)$ in $x$ and substituting (\ref{d-ode-eq})
in the resultant equality, we have
\begin{equation}
f(\hsm(\ttheta), \ttheta)
=
(\hsm)'(\ttheta)
\bigl(
\lambda_2 \ttheta + g(\hsm(\ttheta), \ttheta)
\bigr).
\label{mn03}
\end{equation}
Substituting $(\hsm)'(\ttheta) = 2 s_2 \ttheta + O(\ttheta^2)$,
$g(\hsm(\ttheta),\ttheta) = O(\ttheta^2)$ and
\[
f(\hsm(\ttheta), \ttheta)
=
\frac{(\gamma-1)^2 \kappa^2}{\gamma d} (\pr - \gamma_*)\ttheta^2
+ O(|\ttheta|^3),
\quad
\gamma_*
:=
\frac{1}{2} (\gamma^2 - \gamma + 2) > 1
\]
in (\ref{mn03}), we have 
\[
s_2
=
- \frac{(\gamma-1)^3 \mu \kappa^3}{2 d^2}
(\pr - \gamma_*).
\]
If $\pr = \gamma_*$, that is, $s_2 = 0$,
we also compute $s_3$ in the same way:
\[
s_3
=
\frac{\gamma (\gamma-1)^5 \mu \kappa^4}{6 d^2} > 0.
\]

Summarizing the above observation,  we have
\begin{lemma}
Suppose that $M_+=1$ holds.\\
{\rm (i)}
The local center manifold $\ttheta = \hcm(U) = c_2 U^2 +
c_3 U^3 + O(U^4)$ satisfies
$c_2 \gtreqless 0$ if and only if $\pr \gtreqless 2$.
Especially, if $\pr = 2$, i.e., $c_2 = 0$, 
the coefficient $c_3$ is positive.
\\
{\rm (ii)}
The local stable manifold $U = \hsm(\ttheta) = s_2 \ttheta^2 +
s_3 \ttheta^3 + O(\ttheta^4)$ satisfies
$s_2 \gtreqless 0$ if and only if $\pr \lesseqgtr \gamma_* :=
(\gamma^2-\gamma+2)/2$.
Especially, if $\pr = \gamma_*$, i.e., $s_2 = 0$, 
the coefficient $s_3$ is positive.
\end{lemma}

From the local structure of the invariant manifolds 
in the diagonalized coordinate $(U,\ttheta)$,
we  obtain detailed information on the local structure of invariant manifolds
in the original coordinate $(u,\theta)$.
Let $\theta = \thcm(u)$ and $\theta = \thsm(u)$ be
a local center manifold and a local stable manifold
in the coordinate $(u,\theta)$,
respectively (also see  Figure \ref{fig-mani}).
Then we see that the relations
$\theta = \thcm(u)$ and $\theta = \thsm(u)$ 
are equivalent to
\begin{equation}
\hat{\ttheta} (u,\theta) = \hcm \bigl( \hat{U} (u,\theta) \bigr)
\quad \text{and} \quad
\hat{U} (u,\theta) = \hsm \bigl( \hat{\ttheta} (u,\theta) \bigr),
\label{orig-mani}
\end{equation}
respectively.
Therefore,
substituting (\ref{sm-lg}) and (\ref{loc-mani}) in (\ref{orig-mani})
and solving the resultant equation
with respect to $\theta$,
we get 
\begin{align*}
\thcm(u)
& =
1 + (\gamma-1) (u+1)
+ \frac{\gamma (\gamma-1)}{2(\pr + \gamma - 1)} (\pr - 2)
(u+1)^2
+ O(|u+1|^3),
\\
\thsm(u)
& =
1 - \pr (u+1)
+ \frac{\pr}{2(\pr + \gamma - 1)} (\pr - \gamma_*)
(u+1)^2
+ O(|u+1|^3).
\end{align*}
Especially, if $\pr = 2$
the local center manifold $\theta=\thcm(u)$ satisfies
\[
\thcm(u)
=
1 + (\gamma-1) (u+1)
- \frac{\gamma(\gamma-1)}{\pr + \gamma - 1}
(u+1)^3
+ O(|u+1|^4),
\]
while
the local stable manifold $\theta=\thsm(u)$ satisfies
\[
\thsm(u)
=
1 - \pr (u+1)
+ \frac{\gamma (\gamma-1) \pr}{6(\pr + \gamma - 1)}
(u+1)^3
+ O(|u+1|^4)
\]
if $\pr = \gamma_*$.

\section{Energy estimate}
\label{est}

In this section, we  prove  Theorem \ref{th-sb}.
The crucial point of the proof is a derivation of {\it a priori}
estimates for a perturbation from the stationary solution
\[
(\vp,\psi,\chi)(t,x)
:=
(\rho,u,\theta)(t,x)
-
(\rt,\ut,\ttt)(x)
\]
in the Sobolev space $H^1$.
Using (\ref{d-nse}) and (\ref{ste}), 
we have the system of equations for $(\vp,\psi,\chi)$ as
\begin{subequations}
\label{pt-eq}
\begin{gather}
\vp_t + u \vp_x + \rho \psi_x
=
- (\ut_x \vp + \rt_x \psi),
\label{pt-eq1}
\\
\rho (\psi_t + u \psi_x)
+ \frac{1}{M_+^2} (p - \pt)_x
= \mu \psi_{xx}
- (\rho u - \rt \ut) \ut_x,
\label{pt-eq2}
\\
\frac{\cv}{M_+^2} \rho \chi_t
+ \frac{\cv}{M_+^2} (\rho u \theta_x - \rt \ut \ttt_x)
=
\kappa \chi_{xx}
+ \mu (u_x^2 - \ut_x^2)
- \frac{1}{M_+^2} (p u_x - \pt \ut_x).
\label{pt-eq3}
\end{gather}
\end{subequations}
The initial and the boundary conditions for $(\vp,\psi,\chi)$
follow from \eqref{ic} and \eqref{bc} as
\begin{gather}
(\vp,\psi,\chi)(0,x)
=
(\vp_0,\psi_0,\chi_0)(x)
:=
(\rho_0,u_0,\theta_0)(x)
- (\rt, \ut, \ttt)(x),
\label{pt-ic} 
\\
(\psi,\chi)(t,0)
= (0,0).
\label{pt-bc}
\end{gather}
Hereafter for simplicity, we often use 
the notations $\vvp := (\vp,\psi,\chi)^\trp$ and
$\vvp_0 := (\vp_0,\psi_0,\chi_0)^\trp$.

To show the existence of a solution to the problem (\ref{pt-eq}),
(\ref{pt-ic}) and (\ref{pt-bc}) locally in time,
we define a function space $X(0,T)$, for  $T>0$, by
\begin{align*}
X(0,T)
:=
\bigl\{
(\vp,\psi,\chi)
\; ; \;
&
\vp \in \holh,
\
(\psi, \chi) \in \holp,
\\
&
(\vp,\psi,\chi) 
\in C([0,T] ; H^1 (\R_+)),
\
\vp_x 
\in L^2(0,T ; L^2 (\R_+)),
\\
&
(\psi_x,\chi_x)
\in L^2(0,T ; H^1 (\R_+))
\bigr\},
\end{align*}
where $\sigma \in (0,1)$ is a constant.
We summarize the existence theorem in the following lemma,
which is proved by a standard iteration method.

\begin{lemma}
\label{local-ext}
Suppose that the initial data satisfies
\[
\vp_0 \in \calb^{1+\sigma},
\
(\psi_0,\chi_0) \in \calb^{2+\sigma},
\quad
(\vp_0,\psi_0,\chi_0) \in H^1 (\R_+)
\]
for a certain $\sigma \in (0,1)$
and compatibility conditions of order $0$ and $1$.
Then there exists a positive constant $T_0$, depending
only on $\holx{1+\sigma}{\vp_0}$ and 
$\holx{2+\sigma}{(\psi_0,\chi_0)}$,
such that the problem \eqref{pt-eq}, \eqref{pt-ic} and \eqref{pt-bc}
has a unique solution $(\vp,\psi,\chi) \in X(0,T_0)$.
\end{lemma}

Next we show {\it a priori} estimates of the perturbation 
$(\vp,\psi,\chi)$ in the space $H^1$.
Here 
we utilize the Poincar\'e type inequality 
 in the next lemma.
Since this lemma is proved in the similar way to the 
paper \cite{knz03}, we omit the proof.
\begin{lemma}
For functions
$f \in H^1(\R_+)$ and $w \in L^1_1(\R_+)$, 
we have
\begin{equation}
\int_{\R_+} |w(x) f(x)^2| \, dx
\le
C \lpa{1}{1}{w}
 (f(0)^2 + \lt{f_x}^2).
\label{poin}
\end{equation}
\end{lemma}

To summarize the {\it a priori} estimate,
we define non-negative functions $N(t)$ and $D(t)$ by
\begin{gather*}
N(t)
:=
\sup_{0 \le \tau \le t} \ho{\vvp(\tau)},
\\
D(t)^2
:=
|(\vp,\vp_x)(t,0)|^2
+ \lt{\vp_x(t)}^2
+ \ho{(\psi_x,\chi_x)(t)}^2.
\end{gather*}

\begin{proposition}
\label{pro-1}
Assume that the stationary solution exists.
Namely, one of the following three conditions is supposed to hold:
{\rm (i)} $M_+ > 1$ and \eqref{sm-del}, \ 
{\rm (ii)} $M_+ = 1$ and \eqref{bc-mnz}, or \
{\rm (iii)} $M_+ < 1$ and \eqref{bc-mnn}.
Let $\vvp = (\vp,\psi,\chi) \in X(0,T)$ be a solution to
\eqref{pt-eq}, \eqref{pt-ic} and \eqref{pt-bc}
for a certain constant $T > 0$.
Then there exist positive constants $\ep_3$ and $C$ 
independent of $T$ such that if $N(T) + \dels \le \ep_3$,
then the solution $\vvp$ satisfies the estimate
\begin{equation}
\ho{\vvp(t)}^2
+
\int_0^t D(\tau)^2 \, d \tau
\le
C \ho{\vvp_0}^2.
\label{apri-1}
\end{equation}
\end{proposition}

We  prove Proposition \ref{pro-1} 
in Section \ref{apri-nd}
for the case where
the stationary solution is non-degenerate, that is, $M_+ \neq 1$.
Since the decay property of the degenerate stationary solution for
 the case $M_+ = 1$
is different from that of the non-degenerate one,
we have to modify the derivation of the estimate (\ref{apri-1})
for  $M_+ = 1$.
It will be studied in Section \ref{apri-d}.

In deriving {\it a priori} estimates,
we have to employ a mollifier with respect to time
variable $t$ to resolve an insufficiency  of regularity of 
the solution obtained in Lemma \ref{local-ext}.
As this argument is standard,
we omit  detailed  computations and proceed
a derivation of the estimates as if the solution
verifies the sufficient regularity.

\subsection{Estimates for supersonic and subsonic flows}
\label{apri-nd}

In this section, we obtain the uniform {\it a priori} estimates
of the perturbation from the non-degenerate stationary solution.
Namely, we show (\ref{apri-1}) for the case $M_+ \neq 1$.
In order to obtain the estimate (\ref{apri-1}),
we firstly  derive  a basic $L^2$ estimate.
To this end, it is convenient to employ an energy form $\cale$
defined by
\begin{gather*}
\cale
:=
\frac{1}{M_+^2 \gamma} \ttt \omega \Bigl( \frac{\rt}{\rho} \Bigr)
+ \frac{1}{2} \psi^2
+ \frac{\cv}{M_+^2} \ttt \omega \Bigl( \frac{\theta}{\ttt} \Bigr),
\quad
\omega(s) := s -1 -\log s.
\end{gather*}
Owing to  a smallness assumption on $N(T)$, 
a quantity $\li{\vvp}$ is also sufficiently small.
Hence we see that the energy form is equivalent to $|\vvp|^2$:
\begin{equation}
c \vp^2 
\le \omega \Bigl( \frac{\rt}{\rho} \Bigr)
\le C \vp^2,
\quad
c \chi^2 
\le \omega \Bigl( \frac{\theta}{\ttt} \Bigr)
\le C \chi^2,
\quad
c |\vvp|^2
\le \cale
\le
C |\vvp|^2.
\label{eng-t}
\end{equation}
The solution,  moreover, satisfies the follorin uniform estimates
\begin{equation}
0 < c \le \rho(t,x),\, \theta(t,x) \le C,
\quad
-C \le u(t,x) \le -c < 0
\label{p-rho}
\end{equation}
for $(t,x) \in [0,T] \times \R_+$.

\begin{lemma}
\label{lm-a}
Suppose that $M_+ \neq 1$ and 
the same conditions as in Proposition {\rm \ref{pro-1}}
hold.
Then we have
\begin{gather}
\lt{\vvp(t)}^2
+ \int_0^t
\left(
\vp(\tau,0)^2
+ \lt{(\psi_x,\chi_x)(\tau)}^2
\right)  d \tau
\nonumber
\\
\le
C \lt{\vvp_0}^2
+ C \dels
\int_0^t \lt{\vp_x(\tau)}^2 \, d \tau.
\label{a01}
\end{gather}
\end{lemma}

\begin{proof}
Multiplying (\ref{pt-eq2}) by $\psi$,
and (\ref{pt-eq3}) by $\chi/\theta$,
then adding up the resultant two equalities,
we have
\begin{gather}
(\rho \cale)_t
- (G_1^{(1)} + B_1)_x
+ \mu \frac{\ttt}{\theta} \psi_x^2
+ \kappa \frac{\ttt}{\theta^2} \chi_x^2
=
\ut_x G_1^{(2)} + \ttt_x G_1^{(3)} + R_1,
\label{a02}
\\
G_1^{(1)} :=
- \rho u \cale
- \frac{1}{M_+^2} (p - \pt) \psi,
\quad
B_1 :=
\mu \psi \psi_x + \frac{\kappa}{\theta} \chi \chi_x,
\nonumber
\\
G_1^{(2)} :=
- (\rho u - \rt \ut) \psi
- \frac{1}{M_+^2 \gamma} \vp \chi
+ \frac{1}{M_+^2 \gamma} \frac{\ttt}{\ut} \vp \psi
- \frac{1}{M_+^2 \gamma} \frac{\rt}{\theta} \chi^2,
\nonumber
\\
G_1^{(3)} :=
\frac{1}{M_+^2 \gamma} \rho u \omega \Bigl( \frac{\rt}{\rho} \Bigr)
+ \frac{\cv}{M_+^2} \rho u \omega \Bigl( \frac{\theta}{\ttt} \Bigr)
- \frac{\cv}{M_+^2} 
\frac{1}{\ttt \theta} \chi (\rho u \theta - \rt \ut \ttt),
\nonumber
\\
R_1 :=
\frac{\kappa}{\theta^2} \ttt_x \chi \chi_x
+ \frac{2 \mu}{\theta} \ut_x \chi \psi_x.
\nonumber
\end{gather}
Due to the boundary conditions (\ref{d-bc}) and (\ref{pt-bc}),
the integral of the second term on the left-hand side of (\ref{a02})
is estimated from below as
\begin{equation}
- \int_{\R_+} (G_1^{(1)} + B_1)_x \, d x
= - (\rho u \cale) |_{x=0}
\ge c \vp(t,0)^2.
\label{a03}
\end{equation}
In order to estimate the right-hand side of (\ref{a02}),
we use (\ref{dc-sp}), (\ref{poin}) and the fact 
$| (G_1^{(2)}, G_1^{(3)}) | \le C |\vvp|^2$, 
which follows from (\ref{eng-t}) and (\ref{p-rho}).
Hence we have
\begin{align}
\int_{\R_+} |\ut_x G_1^{(2)} + \ttt_x G_1^{(3)} + R_1| \, dx
& \le
C \dels \lt{(\psi_x,\chi_x)}^2
+ C \dels \int_{\R_+} e^{-cx} |\vvp|^2 \, dx
\nonumber
\\
& \le
C \dels \left(
\vp(t,0)^2 + \lt{\vvp_x}^2
\right).
\label{a04}
\end{align}
Therefore, integrating (\ref{a02}) over $(0,T) \times \R_+$,
substituting (\ref{a03}) and (\ref{a04}) in the
resultant equality, and then letting $\dels$ suitably small,
we obtain the desired inequality (\ref{a01}).
\end{proof}

Our next aim is to get the estimate for the first order derivative
$(\vp_x,\psi_x,\chi_x)$.
To do this, we first  derive the estimate for $\vp_x$.

\begin{lemma}
\label{lm-b}
Suppose that $M_+ \neq 1$ and 
the same conditions as in Proposition {\rm \ref{pro-1}}
hold.
Then we have
\begin{align}
&
\lt{\vp_x(t)}^2
+ \int_0^t \left(
\vp_x(\tau,0)^2
+ \lt{\vp_x(\tau)}^2
\right) d \tau
\nonumber
\\
& \mspace{20mu}
\le
C \ho{\vvp_0}^2
+ C (N(t) + \dels)
\int_0^t D(\tau)^2 \, d \tau.
\label{b01}
\end{align}
\end{lemma}

\begin{proof}
Differentiate (\ref{pt-eq1}) in $x$ to get
\begin{gather}
\vp_{xt} + u \vp_{xx} + \rho \psi_{xx} 
= f_2,
\label{b02}
\\
f_2
:=
-(2 \vp_x \psi_x + 2 \ut_x \vp_x + 2 \rt_x \psi_x
+ \ut_{xx} \vp + \rt_{xx} \psi).
\nonumber
\end{gather}
Multiplying (\ref{b02}) by $\vp_x$ yields
\begin{equation}
\Bigl(
\frac{1}{2} \vp_x^2
\Bigr)_t
+
\Bigl(
\frac{1}{2} u \vp_x^2
\Bigr)_x
=
- \rho \vp_x \psi_{xx}
+ R_2^{(1)},
\quad
R_2^{(1)}
:=
\frac{1}{2} u_x \vp_x^2
+ f_2 \vp_x.
\label{b03}
\end{equation}
On the other hand, multiplying (\ref{pt-eq2}) by $\rho \vp_x$ yields 
\begin{gather}
(\rho^2 \vp_x \psi)_t
- (\rho^2 \vp_t \psi)_x
+ \frac{1}{M_+^2} p \vp_x^2
=
\mu \rho \vp_x \psi_{xx}
+ G_2
+ R_2^{(2)},
\label{b04}
\\
G_2
:=
\rho^3 \psi_x^2
- \frac{1}{M_+^2 \gamma} \rho^2 \vp_x \chi_x,
\nonumber
\\
R_2^{(2)}
:=
- 2 \rho \rt_x \vp_t \psi
+ \rho^2 \psi_x (\ut_x \vp + \rt_x \psi)
- \frac{1}{M_+^2 \gamma} \rho \vp_x (\ttt_x \vp + \rt_x \chi)
- \rho \ut_x \vp_x (\rho u - \rt \ut).
\nonumber
\end{gather}
Successively multiplying (\ref{b03}) by $\mu$ and adding the resultant
equality to (\ref{b04}), we have
\begin{gather}
\Bigl(
\frac{\mu}{2} \vp_x^2 + \rho^2 \vp_x \psi
\Bigr)_t
+\Bigl(
\frac{\mu}{2} u \vp_x^2 - \rho^2 \vp_t \psi
\Bigr)_x
+ \frac{1}{M_+^2} p \vp_x^2
=
G_2 + R_2,
\label{b05}
\\
R_2
:=
\mu R_2^{(1)} + R_2^{(2)}.
\nonumber
\end{gather}

Owing to the outflow boundary condition on $u$
in \eqref{bc},
the integral of the second term on the left-hand side of (\ref{b05})
is estimated from below as
\begin{equation}
\int_{\R_+} \Bigl(
\frac{\mu}{2} u \vp_x^2 - \rho^2 \vp_t \psi
\Bigr)_x \, dx
=
- \frac{\mu}{2} \ub \vp_x(t,0)^2
\ge
c \vp_x(t,0)^2.
\label{b06}
\end{equation}
Hereafter, we denote $\ep$ an arbitrary positive constant and
$C_\ep$ a positive constant depending on $\ep$.
The first term on the right-hand side of (\ref{b05}) is estimated as
\begin{equation}
|G_2|
\le
\ep \vp_x^2
+ C_\ep |(\psi_x,\chi_x)|^2.
\label{b07}
\end{equation}
Since the second term on the right-hand side of (\ref{b05}) is 
estimated as
\begin{equation}
|R_2|
\le
C |\psi_x| \vp_x^2
+ C \dels |(\vp_x,\psi_x)|^2
+ C \dels e^{-cx} |\vvp|^2,
\label{b09}
\end{equation}
we get the estimate for the integral of $R_2$ as
\begin{equation}
\int_{\R_+} |R_2| \, dx
\le
C (N(t) + \dels) D(t)^2.
\label{b08}
\end{equation}
In deriving (\ref{b08}),
 we have used the estimate
\begin{equation*}
\int_{\R_+} |\psi_x| \vp_x^2 \, dx
\le
\li{\psi_x} \lt{\vp_x}^2
\le
C \ho{\psi_x} \lt{\vp_x}^2
\le
C N(t) (\ho{\psi_x}^2 + \lt{\vp_x}^2)
\end{equation*}
to handle the first term on the right-hand side of (\ref{b09})
and 
the Poincar\'e type inequality (\ref{poin}) to
estimate the third term.

Therefore, integrating (\ref{b05}) over $(0,T) \times \R_+$,
substituting (\ref{b06}), (\ref{b07}) and (\ref{b08})
 in the resultant equality
and then letting $\ep$ small,
we obtain
\begin{align*}
\lt{\vp_x}^2
+ \int_0^t \bigl(
\vp_x(\tau,0)^2 + \lt{\vp_x}^2
\bigr) \, d \tau
\le {}
&
C \lt{\vvp_{0}}^2
+ C \ho{\vvp}^2
+ C \int_0^t \lt{(\psi_x,\chi_x)}^2 \, d \tau
\nonumber
\\
& \mspace{10mu}
+ C(N(t) + \dels) 
\int_0^t D(\tau)^2 \, d \tau,
\end{align*}
which yields the desired estimate (\ref{b01}) by substituting
(\ref{a01}) in the second and the third terms on the right-hand side.
These computations complete the proof.
\end{proof}

Next we  estimate $\psi_x$.

\begin{lemma}
\label{lm-c}
Suppose that $M_+ \neq 1$ and 
the same conditions as in Proposition {\rm \ref{pro-1}}
hold.
Then we have
\begin{equation}
\lt{\psi_x(t)}^2
+ \int_0^t \lt{\psi_{xx}(\tau)}^2 \, d \tau
\le
C \ho{\vvp_0}^2 
+ C(N(t) + \dels) \int_0^t D(\tau)^2 \, d \tau.
\label{c01}
\end{equation}
\end{lemma}

\begin{proof}
Multiplying (\ref{pt-eq2}) by $-\psi_{xx}$ gives
\begin{gather}
\Bigl(
\frac{1}{2} \rho \psi_x^2
\Bigr)_t
- (\rho \psi_x \psi_t)_x
+ \mu \psi_{xx}^2
= G_3 + R_3,
\label{c02}
\\
G_3 
:= 
\rho u \psi_x \psi_{xx}
+ \frac{1}{M_+^2 \gamma} (\theta \vp_x + \rho \chi_x) \psi_{xx},
\nonumber
\\
R_3
:=
\frac{1}{M_+^2 \gamma} (\ttt_x \vp + \rt_x \chi) \psi_{xx}
+ \ut_x (\rho u - \rt \ut) \psi_{xx}
- \rho_x \psi_x \psi_t
+ \frac{1}{2} \rho_t \psi_x^2.
\nonumber
\end{gather}
Notice that $G_3$ satisfies
\begin{equation}
|G_3|
\le
\ep \psi_{xx}^2
+ C_\ep |\vvp_x|^2.
\label{c03}
\end{equation}
The term $R_3$ is estimated,
by 
 (\ref{dc-sp}),
 as
\begin{equation*}
|R_3|
\le
C |\psi_x \vvp_x| |(\vvp_x,\psi_{xx})|
+ C \dels |(\vvp_x,\psi_{xx})|^2
+ C \dels e^{-cx} |\vvp|^2.
\end{equation*}
By using (\ref{poin}) and an inequality
\begin{align*}
\int_{\R_+} |\psi_x \vvp_x| |(\vvp_x,\psi_{xx})| \, dx
& \le
\li{\psi_x} \lt{\vvp_x} \lt{(\vvp_x,\psi_{xx})}
\le 
C N(t) \ho{\psi_x}  \lt{(\vvp_x,\psi_{xx})}
\\
& \le
C N(t) \lt{(\vvp_x, \psi_{xx})}^2,
\end{align*}
we have the estimate for the integral of $R_3$ as
\begin{equation}
\int_{\R_+} |R_3| \, dx
\le
C(N(t) + \dels) D(t)^2.
\label{c04}
\end{equation}
Therefore,
integrating (\ref{c02}) over $(0,t) \times \R_+$ and
substituting (\ref{c03}) and (\ref{c04}) in the resultant
equality,
we obtain the desired estimate (\ref{c01}).
\end{proof}

 We finally derive the estimate for $\chi_x$.

\begin{lemma}
\label{lm-d}
Suppose that $M_+ \neq 1$ and 
the same conditions as in Proposition {\rm \ref{pro-1}}
hold.
Then we have
\begin{equation}
\lt{\chi_x(t)}^2
+ \int_0^t \lt{\chi_{xx}(\tau)}^2 \, d \tau
\le
C \ho{\vvp_0}^2 
+ C(N(t) + \dels) \int_0^t D(\tau)^2 \, d \tau.
\label{d01}
\end{equation}
\end{lemma}

\begin{proof}
Multiply (\ref{pt-eq3}) by $-\chi_{xx}$ to get
\begin{gather}
\Bigl(
\frac{\cv}{2 M_+^2} \rho \chi_x^2
\Bigr)_t
- \Bigl(
\frac{\cv}{M_+^2} \rho \chi_x \chi_t
\Bigr)_x
+ \kappa \chi_{xx}^2
=
G_4 + R_4,
\label{d02}
\\
G_4
:=
\frac{\cv}{M_+^2} \rho u \chi_x \chi_{xx}
+ \frac{1}{M_+^2 \gamma} \rho \theta \chi_x \chi_{xx},
\nonumber
\\
\begin{aligned}
R_4
:= &
-\mu (u_x^2 - \ut_x^2) \chi_{xx}
+ \frac{\cv}{M_+^2} \ttt_x (\rho \psi + \ut \vp) \chi_{xx}
+ \frac{1}{M_+^2 \gamma} \ut_x (\rho \chi + \ttt \vp) \chi_{xx}
\\
& 
- \frac{\cv}{M_+^2} \rho_x \chi_x \chi_t
+ \frac{\cv}{2 M_+^2} \rho_t \chi_x^2.
\end{aligned}
\nonumber
\end{gather}
We see that
$G_4$ is estimated as
\begin{equation}
|G_4|
\le
\ep \chi_{xx}^2
+ C_\ep |\vvp_{x}|^2.
\label{d03}
\end{equation}
By a straightforward computation together with utilizing (\ref{dc-sp}), 
we see that $R_4$ satisfies
\begin{equation*}
|R_4|
\le
C |(\psi_x,\chi_x)| |\vvp_x| |(\vvp_x,\chi_{xx})|
+ C \psi_x^2 |\vvp_x|^2
+ C \dels |(\vvp_x,\chi_{xx})|^2
+ C \dels e^{-cx} |\vvp|^2.
\end{equation*}
Integrating the above estimate with the aid of
 using  inequalities
\begin{gather}
\begin{aligned}
\int_{\R_+} |(\psi_x,\chi_x)| |\vvp_x| |(\vvp_x,\chi_{xx})| \, dx
& \le
\li{(\psi_x,\chi_x)} \lt{\vvp_x} \lt{(\vvp_x,\chi_{xx})}
\\
& \le
C N(t) \lt{(\vvp_x,\psi_{xx},\chi_{xx})}^2,
\end{aligned}
\label{d05}
\\
\int_{\R_+} \psi_x^2 |\vvp_x|^2 \, dx
\le
\li{\psi_x}^2 \lt{\vvp_x}^2
\le
C N(t)^2 \ho{\psi_x}^2,
\label{d06}
\end{gather}
we get the estimate for the integral of $R_4$ as
\begin{equation}
\int_{\R_+} |R_4| \, dx
\le
C(N(t) + \dels) D(t)^2.
\label{d04}
\end{equation}
Thus, integrating (\ref{d02}) over $(0,t) \times \R_+$ and
substituting (\ref{d03}) and (\ref{d04}) in the resultant equality,
we obtain the desired estimate (\ref{d01}).
\end{proof}

\begin{proof}%
[\underline{Proof of Proposition {\rm \ref{pro-1}} for $M_+ \neq 1$}]
Summing up the estimates (\ref{b01}), (\ref{c01})
and (\ref{d01}), 
we have the estimate for the first order derivative $\vvp_x$ as
\begin{gather}
\lt{\vvp_x(t)}^2
+ \int_0^t \bigl(
\vp_x(\tau,0)^2
+ \lt{(\vp_x,\psi_{xx},\chi_{xx})(\tau)}^2
\bigr) \, d\tau
\nonumber
\\
\le
C \ho{\vvp_0}^2
+ C (N(t) + \dels) \int_0^t D (\tau)^2 \, d \tau.
\label{dd01}
\end{gather}
Then, adding (\ref{a01}) to (\ref{dd01})
and letting $N(T) + \dels$ suitably small,
we obtain the desired {\it a priori} estimate (\ref{apri-1}).
\end{proof}

\subsection{Estimates for transonic flow}
\label{apri-d}

In this section, we prove Proposition \ref{pro-1}
for the case $M_+ = 1$, where the
stationary solution is degenerate.
To do this, we define a dissipative norm $\tilde{D}(t)$ by
\[
\tilde{D}(t)^2
:=
D(t)^2 
+ \dels^2 \ltat{-2}{\vvp(t)}^2,
\]
where the norm $\ltat{\alpha}{\,\cdot\,}$ is defined by
\[
\ltat{\alpha}{u}
:=
\Bigl(
\int_{\R_+} (1+ \dels x)^\alpha |u(x)|^2 \, dx
\Bigr)^{1/2}.
\]
Using the above notation,
we show the uniform {\it a priori} estimate
\begin{equation}
\ho{\vvp(t)}^2
+ \int_0^t \tilde{D}(\tau)^2 \, d \tau
\le
C \ho{\vvp_0}^2,
\label{apri-1d}
\end{equation}
provided that $N(T)+\dels$ is sufficiently small.
Since the desired estimate (\ref{apri-1}) immediately follows
from (\ref{apri-1d}), it suffices to show the estimate (\ref{apri-1d}),
which is obtained by combining the estimates (\ref{e01})
and (\ref{f01}).
For the case $M_+=1$,
a decay property of the degenerate stationary
solution is worse than the non-degenerate stationary solution.
Therefore,
in deriving $L^2$ estimate of $\vvp$
summarized in Lemma \ref{lm-e},
we have to utilize the precise estimate (\ref{dst-2}) of the
degenerate stationary solution in order 
to estimate the term $\ut_x G_1^{(2)} + \ttt_x G_1^{(3)}$ in (\ref{a02}).

\begin{lemma}
\label{lm-e}
Suppose that $M_+ = 1$ and 
the same conditions as in Proposition {\rm \ref{pro-1}}
hold.
Then we have
\begin{gather}
\lt{\vvp(t)}^2
+ \int_0^t
\left(
\vp(\tau,0)^2
+ \dels^2 \ltat{-2}{\vvp(\tau)}^2
+ \lt{(\psi_x,\chi_x)(\tau)}^2
\right)  d \tau
\nonumber
\\
\le
C \lt{\vvp_0}^2
+ C \dels
\int_0^t \lt{\vp_x(\tau)}^2 \, d \tau.
\label{e01}
\end{gather}
\end{lemma}

\begin{proof}
Notice that
the solution $(\rho,u,\theta)$ satisfies
\begin{equation}
(\rho,u,\theta) = (1,-1,1) + O(N(t) + \dels),
\label{ee07}
\end{equation}
which follows from (\ref{dst-1}) and $\li{\vvp(t)} \le C N(t)$.
Using the property \eqref{ee07}  and (\ref{dst-1}),
we see that $G_1^{(2)}$ is divided into a main quadratic form 
and residue terms as
\begin{equation*}
G_1^{(2)}
= - \psi^2 - \frac{1}{\gamma} \chi^2 + \frac{\gamma-1}{\gamma} \vp \psi
- \frac{1}{\gamma} \vp \chi
+ O(N(t) + \dels) |\vvp|^2.
\end{equation*}
Hence 
we see 
from  the above expression and (\ref{dst-2})
that 
the first term on the right-hand side of (\ref{a02}) satisfies
\begin{align}
\ut_x G_1^{(2)}
= &
- \frac{\gamma+1}{2 d} \zt^2
\Bigl(
\psi^2 + \frac{1}{\gamma} \chi^2 - \frac{\gamma-1}{\gamma} 
 \vp \psi + \frac{1}{\gamma} \vp \chi 
\Bigr)
\nonumber
\\
& 
\mspace{10mu}
+ O(N(t) + \dels) \zt^2 |\vvp|^2
+ O(\dels) e^{-cx} |\vvp|^2.
\label{e05}
\end{align}
By a similar computation, we have
\begin{equation*}
G_1^{(3)}
=
- \frac{1}{2 \gamma} \vp^2 + \frac{\cv}{2} \chi^2
+ \cv \vp \chi - \cv \psi \chi
+ O(N(t) + \dels) |\vvp|^2,
\end{equation*}
where we have also used the fact that
\begin{equation}
\omega(s) 
= \frac{1}{2} (s-1)^2 + O(|s-1|^3).
\label{ee08}
\end{equation}
Therefore, due to (\ref{dst-2}),
the second term on the right-hand side of (\ref{a02})
satisfies
\begin{align}
\ttt_x G_1^{(3)}
= &
- \frac{\gamma+1}{2 d} \zt^2
\Bigl(
\frac{\gamma-1}{2 \gamma} \vp^2 - \frac{1}{2 \gamma} \chi^2
- \frac{1}{\gamma} \vp \chi + \frac{1}{\gamma} \psi \chi
\Bigr)
\nonumber
\\
& 
\mspace{10mu}
+ O(N(t) + \dels) \zt^2 |\vvp|^2
+ O(\dels) e^{-cx} |\vvp|^2.
\label{e06}
\end{align}
Summing up the expressions (\ref{e05}) and (\ref{e06}),
we have
\begin{gather}
\ut_x G_1^{(2)} + \ttt_x G_1^{(3)}
=
- \frac{\gamma+1}{4 \gamma d} \zt^2 F_1 (\vp,\psi,\chi)
+ O(N(t) + \dels) \zt^2 |\vvp|^2
+ O(\dels) e^{-cx} |\vvp|^2,
\label{e02}
\\
F_1 (\vp,\psi,\chi)
:=
(\gamma-1) \vp^2 + 2 \gamma \psi^2 + \chi^2
- 2(\gamma-1) \vp \psi + 2 \psi \chi.
\nonumber
\end{gather}
The quadratic form $F_1(\vp,\psi,\chi)$ is positive definite since
\begin{equation}
F_1 (\vp,\psi,\chi) 
=
(\gamma-1) (\vp-\psi)^2 + (\psi + \chi)^2 + \gamma \psi^2
\ge
c |\vvp|^2.
\label{e03}
\end{equation}
Due to (\ref{dst-2}),
the remaining term $R_1$ is estimated as
\begin{equation}
|R_1|
\le
C \dels ( \zt^2 |\vvp|^2 + |(\psi_x,\chi_x)|^2 + e^{-cx} |\vvp|^2).
\label{e04}
\end{equation}
Therefore, integrating (\ref{a02}) over $(0,t) \times \R_+$,
substituting (\ref{e02}), (\ref{e03}) and (\ref{e04})
in the resultant equality
and letting $N(t) + \dels$ suitably small,
we obtain
\begin{gather*}
\lt{\vvp}^2
+ \int_0^t \bigl( 
\vp(\tau,0)^2 
+ \dels^2 \ltat{-2}{\vvp}^2
+ \lt{(\psi_x,\chi_x)}^2 
\bigr) \, d \tau
\nonumber
\\
\le
C \lt{\vvp_0}^2
+ C \dels \int_0^t \!\! \int_{\R_+} e^{-cx} |\vvp|^2 \, dx \, d \tau,
\end{gather*}
where we have used (\ref{zt-1}).
Finally, to estimate the last term on the right-hand side of the
above inequality,
we utilize the Poincar\'e type inequality (\ref{poin}).
Consequently, we arrive at the desired estimate (\ref{e01})
and complete the proof.
\end{proof}

Next we show the estimate for the first order derivative $\vvp_x$.

\begin{lemma}
\label{lm-f}
Suppose that $M_+ = 1$ and 
the same conditions as in Proposition {\rm \ref{pro-1}}
hold.
Then we have
\begin{gather}
\lt{\vvp_x(t)}^2
+ \int_0^t \bigl(
\vp_x(\tau,0)^2
+ \lt{(\vp_x,\psi_{xx},\chi_{xx})(\tau)}^2
\bigr) \, d\tau
\nonumber
\\
\le
C \ho{\vvp_0}^2
+ C (N(t) + \dels) \int_0^t \tilde{D}(\tau)^2 \, d \tau.
\label{f01}
\end{gather}
\end{lemma}

\begin{proof}
In the present proof, we only show the estimate for the
remained terms $R_2$, $R_3$ and $R_4$.
The other part of the derivation of (\ref{f01}) 
is almost same as that of the non-degenerate case,
so we omit the details.
Using (\ref{dst-3}), we see
\begin{align*}
|(R_2,R_3,R_4)|
\le {}&
C |(\psi_x,\chi_x)| |\vvp_x| |(\vvp_x, \psi_{xx}, \chi_{xx})|
+ C \psi_x^2 |\vvp_x|^2
\\
&
{}+ C \dels \zt^2 |\vvp|^2
+ C \dels |(\vvp_x,\psi_{xx}, \chi_{xx})|^2
+ C \dels e^{-cx} |\vvp|^2.
\end{align*}
By computations similar to (\ref{d05}) and (\ref{d06}),
and by using (\ref{zt-1}), we have
\begin{equation*}
\int_{\R_+} |(R_2,R_3,R_4)| \, dx
\le
C (N(t) + \dels) \tilde{D}(t)^2.
\end{equation*}
Therefore, following the same procedure of the derivation of
(\ref{b01}), (\ref{c01}) and (\ref{d01})
and using the above estimate for the remaining terms,
we obtain the desired estimate (\ref{f01}).
\end{proof}


\subsection{Proof of Theorem {\rm \ref{th-sb}}}
\label{prf-th-sb}

This section is devoted to the proof of Theorem \ref{th-sb}.
Firstly we prove the existence of the solution in the
sense of (\ref{global-sp}) globally in time.
Since the existence time $T_0$ in Lemma \ref{local-ext}
depends on the H\"older norm of the initial data,
we have to show the {\it a priori} estimate in the H\"older norm.
Precisely we prove
\begin{equation}
\holtx{1+\sigma/2}{1+\sigma}{\rho},
\
\holtx{1+\sigma/2}{2+\sigma}{(u,\theta)}
\le
C(T)
\label{hol-est}
\end{equation}
for the solution $(\rho,u,\theta)$ 
satisfying $(\rho,u,\theta) - (\rt,\ut,\ttt) \in X(0,T)$,
where $C(T)$ is a positive constant depending on 
$T$,
$\holx{1+\sigma}{\rho_0}$,
$\holx{2+\sigma}{(u_0,\theta_0)}$
and 
$\ho{(\vp_0,\psi_0,\chi_0)}$.
To obtain the estimate (\ref{hol-est}),
we rewrite the system (\ref{d-nse})  in the
Eulerian coordinate into that in the Lagrangian mass
coordinate,
and then apply the Schauder theory for parabolic equations
studied in \cite{friedman64}
with the aid of the $H^1$ uniform estimate (\ref{apri-1}).
Since the derivation of the H\"older estimate (\ref{hol-est}) is
same as that in \cite{knz03} studying the stability of the
stationary solution for an isentropic model,
we omit the details of the proof.
Therefore, combining Lemma \ref{local-ext} and the estimate (\ref{hol-est})
by using the standard continuation argument,
we obtain the existence of the solution globally in time.
Moreover, we see that the solution verifies
\begin{equation}
\sup_{t \in [0,\infty)} \ho{\vvp(t)}^2
+
\int_0^\infty D(t) \, d t
\le
C
\ho{\vvp_0}^2.
\label{apri-3}
\end{equation}

Next we show the stability (\ref{conv-1}).
For this purpose, it suffices to show that
\begin{equation*}
\lt{\vvp_x(t)}
\to 0
\; \ \text{as} \; \
t \to \infty
\end{equation*}
since we see $\li{\vvp(t)} \le C \lt{\vvp(t)}^{1/2} \lt{\vvp_x(t)}^{1/2}$
and $\lt{\vvp(t)} \le C$ due to the $H^1$ uniform estimate (\ref{apri-1}).
Let $I(t) := \lt{\vp_x(t)}^2$.
By a similar computation to \cite{knz03}, we have
\[
\Bigl|
\frac{d}{dt} I(t)
\Bigr|
\le
C D(t)^2,
\]
which gives $\frac{d}{dt} I \in L^1(0,\infty)$ owing to (\ref{apri-3}).
Combining this fact with $I \in L^1(0,\infty)$, 
which is a direct consequence of (\ref{apri-3}),
we have $I(t) \to 0$, i.e.,
$\lt{\vp_x(t)} \to 0$ as $t \to \infty$.
The convergence $\lt{(\psi_x,\chi_x)(t)} \to 0$ is proved
in the similar computations.
Consequently, 
we prove (\ref{conv-1}) and complete the proof of 
Theorem \ref{th-sb}.

\section{Weighted energy estimate}
\label{w-est}

In this section,
we show the proof of Theorem \ref{th-cv}.
Precisely, we obtain convergence rates of the solution  
toward the stationary solution by using a time and space
weighted energy method.

\subsection{Estimates for supersonic flow}
\label{w-est-nd}

This section is devoted to showing the convergence (\ref{conv-sp})
for the case $M_+ > 1$.
To this end,
we define weighted norm $E_\alpha(t)$ and $D_\alpha(t)$ by
\begin{gather*}
E_\alpha(t)^2
:=
\ho{\vvp(t)}^2
+ \lta{\alpha}{\vvp(t)}^2,
\\
D_\alpha(t)^2
:=
D(t)^2 + \alpha \lta{\alpha-1}{\vvp(t)}^2
+ \lta{\alpha}{(\psi_x,\chi_x)(t)}^2
\end{gather*}
and obtain the weighted energy estimates summarized in
the next proposition.

\begin{proposition}
\label{pro-3}
We assume that $M_+ > 1$ and \eqref{sm-del} hold.
Let $\vvp = (\vp,\psi,\chi) \in X (0,T)$ be a
solution to \eqref{pt-eq}, \eqref{pt-ic} and \eqref{pt-bc}
satisfying $\vvp \in C([0,T] ; L^2_\alpha(\R_+))$
for certain constants $\alpha > 0$ and $T > 0$.
Then there exist positive constant $\ep_4$ and $C$ independent
of $T$ such that if $N(T) + \dels \le \ep_4$,
then the solution $\vvp$ satisfies the following
estimates
\begin{equation}
(1+t)^j E_{\alpha - j}(t)^2
+ \int_0^t (1+\tau)^j D_{\alpha-j}(\tau)^2 \, d \tau
\le
C 
E_\alpha(0)^2,
\label{apri-3a}
\end{equation}
for an arbitrary integer $j = 0,\dots,[\alpha]$
and
\begin{equation}
(1+t)^\xi E_{0}(t)^2
+ \int_0^t (1+\tau)^\xi D_{0}(\tau)^2 \, d \tau
\le
C 
E_\alpha(0)^2
(1+t)^{\xi-\alpha}
\label{apri-3b}
\end{equation}
for an arbitrary $\xi > \alpha$.
\end{proposition}

The proof of Proposition \ref{pro-3} is based on 
the time and space weighted estimate of $\vvp$ in $L^2 (\R_+)$
and the time weighted estimate of $\vvp_x$.

\begin{lemma}
\label{lm-g}
Suppose that the same conditions as in Proposition {\rm \ref{pro-3}}
hold.
Then we have
\begin{align}
 &
(1+t)^\xi \lta{\beta}{\vvp(t)}^2
+ \int_0^t (1+\tau)^\xi \left(
\vp(\tau,0)^2
+ \beta \lta{\beta-1}{\vvp(\tau)}^2
+ \lta{\beta}{(\psi_x,\chi_x)(\tau)}^2
\right) d \tau
\nonumber
\\
& \mspace{20mu}
\le
C \lta{\beta}{\vvp_0}^2
+ C \xi \int_0^t (1+\tau)^{\xi-1} \lta{\beta}{\vvp(\tau)}^2 \, d \tau
+ C \dels \int_0^t (1+\tau)^\xi \lt{\vp_x(\tau)}^2 \, d \tau
\label{g01}
\end{align}
for arbitrary $\beta \in [0,\alpha]$ and $\xi \ge 0$.
\end{lemma}

\begin{proof}
Multiplying (\ref{a02}) by a weight function
$w(t,x) := (1+t)^\xi (1+x)^\beta$,
we have
\begin{gather}
(w \rho \cale)_t
- \bigl\{
w (G_1^{(1)} + B_1)
\bigr\}_x
+ w_x G_1^{(1)}
+ w \Bigl(
\mu \frac{\ttt}{\theta} \psi_x^2
+ \kappa \frac{\ttt}{\theta^2} \chi_x^2
\Bigr)
\nonumber
\\
=
w_t \rho \cale
- w_x B_1
+ w(\ut_x G_1^{(2)} + \ttt_x G_1^{(3)} + R_1).
\label{g02}
\end{gather}
The integral of the second term on the left-hand side of (\ref{g02})
is estimated from below as
\begin{equation}
- \int_{\R_+} 
\bigl\{
w (G_1^{(1)} + B_1)
\bigr\}_x \, dx
\ge
c (1+t)^\xi \vp(t,0)^2.
\label{g03}
\end{equation}
Due to (\ref{dst-1}), (\ref{ee07}) and (\ref{ee08}),
the term $G_1^{(1)}$ is divided into a quadratic form and remaining terms as
\begin{gather}
G_1^{(1)}
=
\frac{1}{2 M_+^2 \gamma (\gamma-1)} F_2 (\vp,\psi,\chi)
+ O(N(t) + \dels) |\vvp|^2,
\label{g04}
\\
F_2 (\vp,\psi,\chi)
:=
(\gamma-1) \vp^2 + M_+^2 \gamma (\gamma-1) \psi^2 + \chi^2
-2 (\gamma-1) (\vp + \chi) \psi.
\label{g10}
\end{gather}
Notice that the quadratic form $F_2$ is positive definite 
owing to the assumption
 $M_+ > 1$
since
\begin{equation}
F_2 (\vp,\psi,\chi)
 =
(\gamma-1) (\vp - \psi)^2
+ \{(\gamma-1) \psi - \chi\}^2
+ \gamma (\gamma-1) (M_+^2-1) \psi^2
\ge
c |\vvp|^2.
\label{g05}
\end{equation}
Thus, substituting (\ref{g05}) in (\ref{g04}),
we have the estimate of the third term on the left-hand side
of (\ref{g02}) from below as
\begin{equation}
\int_{\R_+} w_x G_1^{(1)} \, dx
\ge
\{ c - C (N(t) + \dels) \}
\beta (1+t)^\xi \lta{\beta-1}{\vvp}^2.
\label{g06}
\end{equation}
The first and second terms on the right-hand side of (\ref{g02})
are estimated with the aid of the Schwarz inequality as
\begin{gather}
\int_{\R_+} |w \rho \cale| \, d x
\le
C \xi (1+t)^{\xi-1} \lta{\beta}{\vvp}^2,
\label{g07}
\\
\int_{\R_+} |w_x B_1| \, dx
\le
C \beta (1+t)^\xi 
\bigl(
\ep \lta{\beta-1}{\vvp}^2 
+ C_\ep \lta{\beta-1}{(\psi_x,\chi_x)}^2
\bigr).
\label{g08}
\end{gather}
In the similar way to the derivation of (\ref{a04}),
we estimate the remaining terms in (\ref{g02}),
by using (\ref{dc-sp}) and (\ref{poin}),
 as
\begin{equation}
\int_{\R_+} w |\ut_x G_1^{(2)} + \ttt_x G_1^{(3)} + R_1| \, dx
\le
C \dels (1+t)^\xi (\vp(t,0)^2 + \lt{\vvp_x}^2).
\label{g09}
\end{equation}
Therefore, integrating (\ref{g02}) over $(0,t) \times \R_+$,
substituting (\ref{g03}) and (\ref{g06}) - (\ref{g09})
in the resultant equality
and letting $\ep$ and $N(t) + \dels$ sufficiently small,
we arrive at
\begin{align*}
&
(1+t)^\xi \lta{\beta}{\vvp}^2
+ \int_0^t (1+\tau)^\xi \bigl(
\vp(\tau,0)^2 
+ \beta \lta{\beta-1}{\vvp}^2
+ \lta{\beta}{(\psi_x,\chi_x)}^2
\bigr) \, d \tau
\nonumber
\\
& \mspace{10mu}
\le
C \lta{\beta}{\vvp_0}^2
+ C \xi \int_0^t (1+\tau)^{\xi-1} \lta{\beta}{\vvp}^2 \, d \tau
+ C \dels \int_0^t (1+\tau)^\xi \lt{\vp_x}^2 \, d \tau
\nonumber
\\
& \mspace{40mu}
+ C \beta \int_0^t (1+\tau)^\xi \lta{\beta-1}{(\psi_x,\chi_x)}^2
\, d \tau.
\end{align*}
We finally apply induction with respect to $\beta$ to estimate
the last term on the right-hand side of the above inequality.
This computation yields the desired estimate (\ref{g01}).
Consequently, we complete the proof.
\end{proof}

Letting $\beta=0$ in (\ref{g01}), 
we have the time weighted estimate
\begin{align}
&
(1+t)^\xi \lt{\vvp(t)}^2
+ \int_0^t (1+\tau)^\xi
\bigl(
\vp(\tau,0)^2
+ \lt{(\psi_x,\chi_x)(\tau)}^2
\bigr) \, d \tau
\nonumber
\\
& \mspace{20mu}
\le
C \lt{\vvp_0}^2
+ C \xi \int_0^t (1+\tau)^{\xi-1} \lt{\vvp(\tau)}^2 \, d \tau
+ C \dels \int_0^t (1+\tau)^\xi \lt{\vp_x(\tau)}^2 \, d \tau
\label{h01}
\end{align}
for an arbitrary $\xi \ge 0$.

We state below the time weighted estimate for the first order
derivative $\vvp_x$.
Since the proof of this estimate is almost same as that of (\ref{dd01}),
we omit the details 
and only summarize the result in the next lemma.

\begin{lemma}
\label{lm-h}
Suppose that the same conditions as in Proposition {\rm \ref{pro-3}}
hold.
Then we have
\begin{align}
 &
(1+t)^\xi \lt{\vvp_x(t)}^2
+ \int_0^t (1+\tau)^\xi
\left(
\vp_x(\tau,0)^2
+ \lt{(\vp_x,\psi_{xx},\chi_{xx})(\tau)}^2
\right) d \tau
\nonumber
\\
& \mspace{30mu}
\le
C \ho{\vvp_0}^2
+ C \xi \int_0^t (1+\tau)^{\xi-1} \ho{\vvp(\tau)}^2 \, d \tau
\nonumber
\\
& \mspace{80mu}
+ C (N(t) + \dels) \int_0^t (1+\tau)^\xi D(\tau)^2 \, d \tau
\label{h02}
\end{align}
for an arbitrary $\xi \ge 0$.
\end{lemma}

We conclude this section by giving the proofs of Proposition \ref{pro-3}
and Theorem \ref{th-cv} - (i).

\begin{proof}%
[\underline{Proofs of Proposition {\rm \ref{pro-3}} 
and Theorem {\rm \ref{th-cv} - (i)}}]
Summing up the inequalities (\ref{h01}) and (\ref{h02}),
we have the time weighted $H^1$ estimate
\begin{align*}
&
(1+t)^\xi \ho{\vvp(t)}^2
+ \int_0^t (1+\tau)^\xi D(\tau)^2 \, d \tau
\nonumber
\\
& \mspace{20mu}
\le
C \ho{\vvp_0}^2
+ C \xi \int_0^t (1+t)^{\xi-1} \ho{\vvp(\tau)}^2 \, d \tau.
\end{align*}
Add (\ref{g01}) to the above inequality to obtain
\begin{align*}
&
(1+t)^\xi E_\beta(t)^2
+ \int_0^t (1+\tau)^\xi 
 \bigl( \beta \lta{\beta-1}{\vvp(\tau)}^2 + D_\beta(\tau)^2 \bigr) \, d \tau
\nonumber
\\
& \mspace{20mu}
\le C E_\beta(0)^2
+ C \xi \int_0^t (1+\tau)^{\xi-1} \bigl(
\lta{\beta}{\vvp(\tau)}^2
+ D_\beta(\tau)^2
\bigr) \, d \tau,
\end{align*}
where we have used the inequalities
\begin{equation*}
\beta \lta{\beta-1}{\vvp(t)}^2
+ D_\beta(t)^2
\le
2 D_\beta(t)^2,
\quad
\ho{\vvp(t)}^2 + \lta{\beta}{\vvp(t)}^2
\le
2 \lta{\beta}{\vvp(t)}^2
+ D_\beta(t)^2.
\end{equation*}
By applying induction with respect to $\beta$ and $\xi$,
studied by \cite{km-85} and \cite{ns-98},
we have the desired estimates (\ref{apri-3a}) and (\ref{apri-3b}).
The convergence (\ref{conv-sp})  immediately follows from 
 (\ref{apri-3b}) and the Sobolev inequality.
Consequently, we complete the proofs of Proposition \ref{pro-3}
and Theorem \ref{th-cv} - (i).
\end{proof}

\subsection{Estimates for transonic flow}
\label{w-est-d}

In this section, we show the convergence (\ref{conv-tr})
for the case $M_+=1$ by deriving the time and space weighted
estimate in $H^1$.
To do this, we define weighted norms by
\begin{gather*}
\tilde{N}_\alpha (t)
:=
\sup_{0 \le \tau \le t} 
\tilde{E}_\alpha(\tau),
\quad
\tilde{E}_\alpha(t)
:=
\hoat{\alpha}{\vvp(t)},
\\
\tilde{D}_\alpha (t)^2
:=
|(\vp,\vp_x)(t,0)|^2
+ \dels^2 \ltat{\alpha-2}{\vvp(t)}^2
+ \ltat{\alpha}{\vp_x(t)}^2
+ \hoat{\alpha}{(\psi_x,\chi_x)(t)}^2,
\end{gather*}
where $\hsat{s}{\alpha}{\cdot}$ is the $s$-th order Sobolev norm
corresponding to $\ltat{\alpha}{\cdot}$:
\[
\hsat{s}{\alpha}{u}
:=
\Bigl(
\sum_{k=0}^s
\ltat{\alpha}{\pd_x^k u}^2
\Bigr)^{1/2}
=
\Bigl(
\sum_{k=0}^s
\int_{\R_+} (1+\dels x)^\alpha |\pd_x^k u(x)|^2 \, dx
\Bigr)^{1/2}.
\]

\begin{proposition}
\label{pro-4}
We assume that $M_+=1$ and \eqref{bc-mnz} hold.
Let $\vvp \in X(0,T)$ be a solution to \eqref{pt-eq}, \eqref{pt-ic}
and \eqref{pt-bc} satisfying $\vvp \in C ([0,T] ; H^1_\alpha(\R_+))$
for certain constants $\alpha \in [1,2(1+\sqrt{2}))$ and $T > 0$.
Then there exist positive constants $\ep_4$ and $C$ independent of $T$
such that if $\dels^{-1/2} \tilde{N}_\alpha(T) + \dels \le \ep_4$,
then the solution $\vvp$ satisfies the following estimates
for $t \in [0,T]$:
\begin{equation}
(1+t)^j \tilde{E}_{\alpha - 2j}(t)^2
+ \int_0^t (1+\tau)^j \tilde{D}_{\alpha-2j}(\tau)^2 \, d \tau
\le
C \dels^{-2j} 
\tilde{E}_\alpha(0)^2
\label{apri-4a}
\end{equation}
for an arbitrary integer $j=0,\dots,[\alpha/2]$ and
\begin{equation}
(1+t)^\xi 
\tilde{E}_0(t)^2
+ \int_0^t (1+\tau)^{\xi} \tilde{D}_0(\tau)^2 \, d \tau
\le
C \dels^{-\alpha}
\tilde{E}_\alpha(0)^2
(1+t)^{\xi-\alpha/2}
\label{apri-4b}
\end{equation}
for an arbitrary $\xi > \alpha/2$.
\end{proposition}

In order to prove Proposition \ref{pro-4},
we have to derive time and space weighted estimates
not only for $\vvp$ in $L^2$ but also for the first order 
derivative $\vvp_x$.
In deriving the weighted estimate for $\vvp$ in $L^2$,
we utilize the following interpolation inequality
to handle several nonlinear terms.

\begin{lemma}
\label{lm-i}
Let $\beta \ge 1$.
Then a function 
$f \in H^1_\beta (\R_+)$ satisfies
\begin{equation}
\int_{\R_+} (1+\dels x)^{\beta-1} |f(x)|^3 \, dx
\le
C \dels^{-3/2}
\ltat{1}{f} \bigl(
f(0)^2 
+  \dels^2 \ltat{\beta-2}{f}^2
+ \ltat{\beta}{f_x}^2
\bigr).
\label{i01}
\end{equation}
\end{lemma}

Since we can prove (\ref{i01}) in the similar way to
the paper \cite{unk08},
we omit the proof of Lemma \ref{lm-i}.
For details, see Lemma 5.1 with $p=2$ and $\alpha=\beta$ in \cite{unk08}.

Then
we  show the time and space weighted $L^2$ estimate.
In deriving this estimate,
we have to assume that the weight exponent $\alpha$ is less than
$2(1+\sqrt{2})$ in order to obtain the dissipative term 
$\dels^2 \ltat{\beta-2}{\vvp}^2$.
Moreover, to control the term $\zt^{\beta+1} |\vvp|^3$ in (\ref{j04}),
we have to assume the smallness of $\ltat{1}{\vvp}$.
Hence we need a condition $\alpha \ge 1$, too.

\begin{lemma}
\label{lm-j}
Suppose that the same conditions as in Proposition {\rm \ref{pro-4}}
hold.
Then we have
\begin{align}
&
(1+t)^\xi \ltat{\beta}{\vvp(t)}^2
+ \int_0^t (1+\tau)^\xi \bigl(
\vp(\tau,0)^2
+ \dels^2 \ltat{\beta-2}{\vvp(\tau)}^2
+ \ltat{\beta}{(\psi_x,\chi_x)(\tau)}^2
\bigr) \, d \tau
\nonumber
\\
& \mspace{5mu}
\le
C \ltat{\beta}{\vvp_0}^2
+ C \xi \int_0^t (1+\tau)^{\xi-1} \ltat{\beta}{\vvp(\tau)}^2 \, d \tau
+ C ( \dels^{-1/2} \tilde{N}_\beta(t)  + \dels)
\int_0^t (1+\tau)^\xi \tilde{D}_\beta(\tau)^2 \, d \tau
\label{j01}
\end{align}
for arbitrary constants $\beta \in [1,\alpha]$ and $\xi \ge 0$.
\end{lemma}

\begin{proof}
In the present proof,
we employ a spatial weight function
\[
w(x)
:=
h^2 \zt(x)^{-\beta},
\quad
h := \Bigl( \frac{4d}{\gamma+1} \Bigr)^{1/2}.
\]
Notice that $w \sim \dels^{-\beta} (1+\dels x)^\beta$ holds
due to (\ref{zt-1}).
Multiplying (\ref{a02}) by the weight function $w(x)$,
we have
\begin{gather}
(\rho w \cale)_t
- \bigl\{
w (G_1^{(1)} + B_1)
\bigr\}_x
+ w_x G_1^{(1)}
+ w_x B_1
+ w \Bigl(
\mu \frac{\ttt}{\theta} \psi_x^2
+ \kappa \frac{\ttt}{\theta^2} \chi_x^2
\Bigr)
\nonumber
\\
=
w (\ut_x G_1^{(2)} + \ttt_x G_1^{(3)})
+ w R_1.
\label{j02}
\end{gather}
The remainder of the present proof is divided into
three steps.

\noindent
\underline{\bf Step 1}.
Firstly we show that the equality (\ref{j02}) is rewritten as
\begin{equation}
(\rho w \cale)_t
- \bigl\{ w (G_1^{(1)} + B_1) \bigr\}_x
+ \tilde{F}
=
\tilde{R},
\label{j03}
\end{equation}
where $\tilde{F}$ is defined by
\begin{gather}
\tilde{F}
:=
\frac{\cv}{2} w_x F_2
+ \zt^{-\beta+2} F_3
+ \mu \Bigl( 
h \psi_x + \frac{\beta}{h} \zt \psi
\Bigr)^2
\zt^{-\beta}
+ \kappa \Bigl(
h \chi_x + \frac{\beta}{h} \zt \chi
\Bigr)^2
\zt^{-\beta},
\nonumber
\\
F_3 = F_3(\vp,\psi,\chi)
:=
\frac{1}{\gamma} F_1(\vp,\psi,\chi)
+ \frac{\beta}{\gamma} F_4(\vp,\psi,\chi)
- \frac{\beta^2}{h^2}
( \mu \psi^2 + \kappa \chi^2),
\nonumber
\\
F_4 = F_4(\vp,\psi,\chi)
:=
(3 - \gamma) \vp^2 + \chi^2 + 2(\gamma-1) \vp \psi
+ 2 \psi \chi
\nonumber
\end{gather}
and the remaining term $\tilde{R}$ satisfies
\begin{align}
|\tilde{R} |
\le {} &
C (N_\beta(t) + \dels) \bigl(
\zt^{-\beta +2} |\vvp|^2
+ \zt^{-\beta} |(\psi_x,\chi_x)|^2
\bigr)
+ C \dels e^{-cx} \zt^{-\beta} |\vvp|^2
\nonumber
\\
& \mspace{20mu}
+ C \zt^{-\beta+1} |\vvp|^3.
\label{j04}
\end{align}
For this purpose, 
we  show that the third term on the left-hand side
of (\ref{j02}) verifies a decomposition
\begin{equation}
w_x G_1^{(1)}
=
\frac{\cv}{2} w_x F_2
+ \frac{1}{\gamma} \beta \zt^{-\beta+2} F_4
+ O(|\vvp| + \zt^2 + \dels e^{-cx}) \zt^{-\beta+1} |\vvp|^2,
\label{j05}
\end{equation}
where $F_2$ is defined in (\ref{g10}).
Using the fact that
\begin{equation}
(\rho, u, \theta)
=
(1,-1,1)
+ (-1, -1, 1-\gamma) \zt
+ O(|\vvp| + \zt^2 + \dels e^{-cx}),
\label{j08}
\end{equation}
which follows from (\ref{dst-1}),
we see that the terms in $G_1^{(1)}$ satisfy
\begin{gather*}
-\rho u \cale
= 
\frac{1}{2 \gamma} \vp^2
+ \frac{1}{2} \psi^2
+ \frac{\cv}{2} \chi^2
+ \Bigl(
\frac{3-\gamma}{2 \gamma} \vp^2
+ \frac{1}{2 \gamma} \chi^2
\Bigr) \zt
+ O(|\vvp| + \zt^2 + \dels e^{-cx}) |\vvp|^2,
\\
- (p - \pt) \psi
 =
-\frac{1}{\gamma} \vp \psi
- \frac{1}{\gamma} \psi \chi
+ \Bigl(
\frac{\gamma-1}{\gamma} \vp \psi
+ \frac{1}{\gamma} \psi \chi
\Bigr) \zt
+ O(|\vvp| + \zt^2 + \dels e^{-cx}) |\vvp|^2.
\end{gather*}
Summing up the above two equalities,
we see that $G_1^{(1)}$ satisfies
\begin{equation}
G_1^{(1)}
=
\frac{\cv}{2} F_2
+ \frac{1}{2 \gamma} \zt F_4
+ O(|\vvp| + \zt^2 + \dels e^{-cx}) |\vvp|^2.
\label{j06}
\end{equation}
Furthermore, by differentiating the weight function $w(x)$
and using (\ref{eq-z2}), 
we have
\begin{equation}
w_x
=
2 \beta \zt^{-\beta+1}
+ O(\beta \zt^{-\beta+2}).
\label{j07}
\end{equation}
Multiplying (\ref{j06}) by (\ref{j07}) yields the
desired equality (\ref{j05}).
We also see that the fourth and the fifth terms on the
left-hand side of (\ref{j02}) are rewritten as 
\begin{align}
&
w_x B_1
+ w \Bigl(
\mu \frac{\ttt}{\theta} \psi_x^2
+ \kappa \frac{\ttt}{\theta^2} \chi_x^2
\Bigr)
\nonumber
\\
& \mspace{10mu}
=
\mu \Bigl(
h \psi_x + \frac{\beta}{h} \zt \psi
\Bigr)^2 \zt^{-\beta}
+ \kappa \Bigl(
h \chi_x + \frac{\beta}{h} \zt \chi
\Bigr)^2 \zt^{-\beta}
- \frac{\beta^2}{h^2}
(
\mu \psi^2
+ \kappa \chi^2
) \zt^{-\beta+2}
\nonumber
\\
& \mspace{30mu}
+ O (|\vvp| + \zt) 
\bigl( \zt^{-\beta+2} |\vvp|^2 
+ \zt^{-\beta} |(\psi_x,\chi_x)|^2
\bigr),
\label{j11}
\end{align}
which follows from (\ref{dst-1}), (\ref{j08}) and (\ref{j07}).
To estimate the right-hand side of (\ref{j02}),
we use (\ref{e02}) and (\ref{e04}) 
to obtain
\begin{gather}
\begin{aligned}
w (\ut_x G_1^{(2)} + \ttt_x G_1^{(3)})
= &
- \frac{1}{\gamma} \zt^{-\beta+2} F_1(\vp,\psi,\chi)
+ O(N(t) + \dels) \zt^{-\beta+2} |\vvp|^2
\\
&
+ O(\dels) e^{-cx} \zt^{-\beta} |\vvp|^2,
\end{aligned}
\label{j09}
\\
|w R_1|
\le
C \dels
\bigl(
\zt^{-\beta+2} |\vvp|^2
+ \zt^{-\beta} |(\psi_x,\chi_x)|^2
+ e^{-cx} \zt^{-\beta} |\vvp|^2
\bigr).
\label{j10}
\end{gather}
Therefore, substituting (\ref{j05}), (\ref{j11}), (\ref{j09})
and (\ref{j10}) in (\ref{j02}),
we obtain the desired equality (\ref{j03}).

\smallskip

\noindent
\underline{\bf Step 2}.
Our next aim is to show that $\tilde{F}$ satisfies
the estimate from below as
\begin{equation}
\tilde{F}
\ge
c \zt^{-\beta+2} |\vvp|^2
+ c \zt^{-\beta} |(\psi_x,\chi_x)|^2
\label{j15}
\end{equation}
provided that $\beta \in [0,2(1+\sqrt{2}))$.
Let $A_2$ be a real symmetric matrix satisfying $F_2 = \vvp^\trp A_2 \vvp$,
i.e.,
\begin{equation*}
A_2
:=
\begin{pmatrix}
\gamma - 1 & 1 - \gamma & 0
\\
1 - \gamma & \gamma (\gamma -1) & 1 - \gamma
\\
0 & 1 - \gamma & 1
\end{pmatrix}.
\end{equation*}
We see that the matrix $A_2$ admits three distinct eigenvalues
$0$, $\nu_-$ and $\nu_+$ satisfying
\[
\nu_\pm
=
\frac{1}{2} \bigl(
\gamma^2 
\pm 
\sqrt{\gamma^4 - 4 \gamma^3 + 12 \gamma^2 - 20 \gamma + 12}
\bigr)
\ \; \text{and} \ \;
0 < \nu_- < \nu_+.
\]
Let $q_1$, $q_2$ and $q_3$ be unit eigenvectors of $A_2$ corresponding to
the eigenvalues $0$, $\nu_-$ and $\nu_+$, respectively.
Especially, we obtain
\[
q_1 = (1, 1, \gamma-1)^\trp \bar{q}
\ \; \text{where} \ \;
\bar{q} := (\gamma^2 - 2 \gamma + 3)^{-1/2}.
\]
Furthermore, we employ a new function $\vvph$ defined by
\[
\vvph
:= (\vph, \psih, \chih)^\trp
:= Q^{-1} \vvp,
\]
where $Q := (q_1, q_2, q_3)$ is an orthogonal matrix.
Using the fact that $Q^\trp A_2 Q = Q^{-1} A_2 Q = \diag (0,\nu_1,\nu_2)$,
we see that the quadratic form $F_2$ satisfies the estimate
from below as
\begin{equation*}
F_2
= (Q \vvph)^\trp A_2 Q \vvph
= \nu_- \psih^2 + \nu_+ \chih^2
\ge
c |(\psih, \chih)|^2.
\end{equation*}
Combining this estimate with the inequality $w_x \ge c \beta \zt^{-\beta+1}$, 
which follows from (\ref{j07}) with $\dels \ll 1$,
we have
\begin{equation}
\frac{\cv}{2} w_x F_2
\ge
c \beta \zt^{-\beta+1} |(\psih, \chih)|^2.
\label{j13}
\end{equation}

Next we employ a real symmetric matrix $A_3$ satisfying 
 $F_3 = \vvp^\trp A_3 \vvp$.
Let $\hat{A}_3 := (\hat{a}_{ij})_{ij} := Q^\trp A_3 Q$.
Then we see that
\begin{equation}
F_3
=
(Q \vvph)^\trp A_3 Q \vvph
=
\vvph^\trp \hat{A}_3 \vvph
=
\hat{a}_{11} \vph^2 
+ O(|(\psih,\chih)|^2 + |\vph (\psih + \chih)|).
\label{j14}
\end{equation}
Since the sign of $\hat{a}_{11}$ will play an important role later,
we obtain it explicitly:
\begin{equation}
\hat{a}_{11}
=
q_1^\trp A_3 q_1
=
F_3|_{\vvp = q_1}
=
\frac{\gamma+1}{4} \bar{q}^2 (4 + 4 \beta - \beta^2).
\label{j12}
\end{equation}
Owing to the above observations,
we show that the first and the second terms in the definition of 
$\tilde{F}$ satisfy
\begin{equation}
\frac{\cv}{2} w_x F_2 
+ \zt^{-\beta+2} F_3
\ge
c \zt^{-\beta+2} |\vvp|^2
\label{jj02}
\end{equation}
provided that $\beta \in [0,2(1+\sqrt{2}))$.
Notice that the estimate (\ref{jj02}) immediately yields the
desired estimate (\ref{j15}).
If  $\beta = 0$, 
the quadratic form $F_3$ is positive definite, i.e., $F_3 \ge c |\vvp|^2$
since we have $F_3 = F_1 / \gamma$ and the positivity of $F_1$
due to (\ref{e03}).
Thus, owing to the continuous dependency on $\beta$,
there exists a positive constant $\beta_*$ such that
$F_3 \ge c |\vvp|^2$ holds for $\beta \in [0,\beta_*]$, 
where $c$ is independent of $\beta$.
Namely, (\ref{jj02}) holds for $\beta \in [0,\beta_*]$.

Next, we show (\ref{jj02}) for $\beta \in [\beta_*,2(1+\sqrt{2}))$.
Note that the constant $\hat{a}_{11}$ is
positive due to (\ref{j12}).
Thus, using (\ref{j13}) and (\ref{j14}),
we have 
\begin{align*}
&
\frac{\cv}{2} w_x F_2
+ \zt^{-\beta+2} F_3
\\
& \mspace{10mu}
\ge
c \beta_* \zt^{-\beta+1} |(\psih,\chih)|^2
+
\hat{a}_{11} \zt^{-\beta+2} \vph^2 
- C \zt^{-\beta+2} \bigl(
|(\psih,\chih)|^2 + |\vph (\psih + \chih)|
\bigr)
\\
& \mspace{10mu}
\ge
\bigl( c \beta_* - C \sqrt{\dels} \bigr)
 \zt^{-\beta+1} |(\psih,\chih)|^2
+ \bigl( \hat{a}_{11} - C \sqrt{\dels} \bigr)
 \zt^{-\beta+2} \vph^2,
\end{align*}
which yields (\ref{jj02}) if $\dels$ is sufficiently small.
Therefore, we have shown that the estimate (\ref{jj02}) holds
for $\beta \in [0,2(1+\sqrt{2}))$.

\smallskip

\noindent
\underline{\bf Step 3}.
Finally we prove (\ref{j01}) by using (\ref{j03}) and (\ref{j15}).
Using (\ref{zt-1}), we have the estimate for the integral of
the second term on the left-hand side of (\ref{j03}) as
\begin{equation}
- \int_{\R_+} \bigl\{
w (G_1^{(1)} + B_1)
\bigr\}_x \, d x
\ge
c \dels^{-\beta} \vp(t,0)^2.
\label{j16}
\end{equation}
Furthermore, due to the Poincar\'e type inequality (\ref{poin})
and the inequality (\ref{i01}),
integrating (\ref{j04}) yields
\begin{equation}
\int_{\R_+} |\tilde{R} | \, dx
\le
C \dels^{-\beta} (\dels^{-1/2} \tilde{N}_\beta(t) + \dels)
\tilde{D}_\beta(t)^2,
\label{j17}
\end{equation}
where we have used $\ltat{1}{\vvp(t)} \le \tilde{N}_\beta(t)$ 
for $\beta \ge 1$.
Consequently, integrating (\ref{j03}) and
substituting (\ref{j15}), (\ref{j16}) and (\ref{j17}) in
the resultant equality gives the desired estimate
(\ref{j01}).
Thus we complete the proof.
\end{proof}

Next we show the estimate for the first order derivative $\vvp_x$.
Owing to the degenerate property of the transonic flow,
we have to employ the spatially weighted energy method
for the estimate for $\vvp_x$.

\begin{lemma}
\label{lm-k}
Suppose that the same conditions as in Proposition {\rm \ref{pro-4}}
hold.
Then we have
\begin{align}
&
(1+t)^\xi \ltat{\beta}{\vvp_x(t)}^2
+ \int_0^t (1+\tau)^\xi \bigl(
\vp_x(\tau,0)^2
+ \ltat{\beta}{\vp_x(\tau)}^2
+ \ltat{\beta}{(\psi_{xx},\chi_{xx})(\tau)}^2
\bigr) \, d \tau
\nonumber
\\
& \mspace{20mu}
\le
C \hoat{\beta}{\vvp_0}^2
+ C \xi \int_0^t (1+\tau)^{\xi-1} \hoat{\beta}{\vvp(\tau)}^2 \, d \tau
\nonumber
\\
& \mspace{60mu}
+ C (\dels^{-1/2} \tilde{N}_\beta(t) + \dels)
  \int_0^t (1+\tau)^\xi \tilde{D}_\beta(\tau)^2 \, d \tau
\label{k01}
\end{align}
for arbitrary constants $\beta \in [1,\alpha]$ and $\xi \ge 0$.
\end{lemma}

\begin{proof}
We only show the estimate for $\vp_x$ as 
the other estimates for $(\psi_x,\chi_x)$ 
can be established by similar computations.
Multiplying (\ref{b05}) by a spatial weight function
$w := (1+\dels x)^\beta$,
we get
\begin{gather}
\Bigl\{
w \Bigl(
\frac{\mu}{2} \vp_x^2
+ \rho^2 \vp_x \psi
\Bigr)
\Bigr\}_t
+
\Bigl\{
w \Bigl(
\frac{\mu}{2} u \vp_x^2
- \rho^2 \vp_t \psi
\Bigr)
\Bigr\}_x
+ w p \vp_x^2
\nonumber
\\
=
w G_2
+ w R_2
+ w_x \Bigl(
\frac{\mu}{2} u \vp_x^2 - \rho^2 \vp_t \psi
\Bigr).
\label{k02}
\end{gather}
The last term on the right-hand side of (\ref{k02})
is estimated  as
\begin{equation}
\int_{\R_+} w_x \Bigl|
\frac{\mu}{2} u \vp_x^2 - \rho^2 \vp_t \psi
\Bigr| \, dx
\le
\ep \ltat{\beta}{\vp_x}^2
+ C_\ep \bigl(
\dels^2 \ltat{\beta-2}{\vvp}^2
+ \ltat{\beta}{\psi_x}^2
\bigr)
+ C \dels \tilde{D}_\beta(t)^2,
\label{k03}
\end{equation}
where $\ep > 0$ is
 an arbitrary constant and 
$C_\ep$ is a positive constant depending on $\ep$.
The other terms in (\ref{k02}) are estimated in a same 
way as the proof of Lemma \ref{lm-b}.
For instance, 
the remaining term $R_2$ verifies the estimate
\begin{equation}
\int_{\R_+} w
 |(R_2,R_3,R_4)| \, dx
\le
C (\tilde{N}_\beta(t) + \dels) \tilde{D}_\beta(t)^2,
\label{k04}
\end{equation}
which follows from the Poincar\'e type inequality (\ref{poin})
and the inequality
\begin{align*}
\int_{\R_+} (1+\dels x)^\beta |(\psi_x,\chi_x)| |\vvp_x|
|(\vvp_x,\psi_{xx},\chi_{xx})| \, dx
& \le
\li{(\psi_x,\chi_x)} \ltat{\beta}{\vvp_x}
  \ltat{\beta}{(\vvp_x,\psi_{xx},\chi_{xx})}
\\
& \le
C \tilde{N}_\beta(t) 
  \ltat{\beta}{(\vvp_x,\psi_{xx},\chi_{xx})}^2.
\end{align*}
Therefore,
integrating (\ref{k02}) over $(0,t) \times \R_+$, substituting
(\ref{b07}), (\ref{k03}) and (\ref{k04}) in the resultant
and then letting $\ep$ sufficiently small
with using (\ref{j01}),
we obtain the estimate for $\vp_x$ in \eqref{k01}.
The estimates for $(\psi_x,\chi_x)$ are obtained by similar
computations to Lemma \ref{lm-c} and \ref{lm-d} with
using the estimate (\ref{k04}).
Thus we complete the proof of the desired estimate (\ref{k01}).
\end{proof}

\begin{proof}%
[\underline{Proofs of Proposition {\rm \ref{pro-4}} 
 and Theorem {\rm \ref{th-cv} - (ii)}}]
Summing up the estimates (\ref{j01}) and (\ref{k01}),
and letting $\dels^{-1/2} \tilde{N}_\beta(t) + \dels$ be suitably small,
we have
\begin{align*}
&
(1+t)^\xi \hoat{\beta}{\vvp(t)}^2
+
\int_0^t (1+\tau)^\xi 
\bigl(
\dels^2 \ltat{\beta-2}{\vvp(\tau)}^2
+ \tilde{D}_{\beta}(\tau)^2
\bigr)
 \, d \tau
\\
&
\mspace{20mu}
\le
C \hoat{\beta}{\vvp_0}^2
+ C \xi \int_0^t (1+\tau)^{\xi-1}
\bigl(
\ltat{\beta}{\vvp(\tau)}^2
+ \tilde{D}_{\beta}(\tau)^2
\bigr)
\, d \tau,
\end{align*}
which yields the desired estimates (\ref{apri-4a}) and (\ref{apri-4b})
by an induction with respect to $\beta$ and $\xi$ (see \cite{{ns-98}}).
The convergence rate (\ref{conv-tr})   follows from
the estimate (\ref{apri-4b}) with the aid of the estimate 
$\dels^\alpha \hoa{\alpha}{\vvp}^2 \le \hoat{\alpha}{\vvp}^2
 \le \hoa{\alpha}{\vvp}^2$.
We consequently  complete the proofs.
\end{proof}


{\small 


\begin{thebibliography}{10}

\bibitem{aoki-91}
{\sc K.~Aoki, K.~Nishino, Y.~Sone, and H.~Sugimoto}, {\em Numerical analysis of
  steady flows of a gas condensing on or evaporating from its plane condensed
  phase on the basis of kinetic theory: Effect of gas motion along the
  condensed phase}, Phys. Fluids A, 3 (1991), pp.~2260--2275.

\bibitem{aoki-90}
{\sc K.~Aoki, Y.~Sone, and T.~Yamada}, {\em Numerical analysis of gas flows
  condensing on its plane condensed phase on the basis of kinetic theory},
  Phys. Fluids A, 2 (1990), pp.~1867--1878.

\bibitem{carr}
{\sc J.~Carr}, {\em Applications of centre manifold theory}, Springer Verlag,
  1981.

\bibitem{friedman64}
{\sc A.~Friedman}, {\em Partial differential equations of parabolic type},
  Prentice Hall, 1964.

\bibitem{oleinik}
{\sc A.~M. Il'in and O.~A. Ole{\u\i}nik}, {\em Behavior of solutions of the
  {C}auchy problem for certain quasilinear equations for unbounded increase of
  the time}, Dokl. Akad. Nauk SSSR, 120 (1958), pp.~25--28.

\bibitem{kg06}
{\sc Y.~Kagei and S.~Kawashima}, {\em Stability of planar stationary solutions
  to the compressible {N}avier-{S}tokes equation on the half space}, Comm.
  Math. Phys., 266 (2006), pp.~401--430.

\bibitem{kur08}
{\sc S.~Kawashima and K.~Kurata}, {\em {H}ardy type inequality and application
  to the stability of degenerate stationary waves}, J. Func. Anal., 257 (2009),
  pp.~1--19.

\bibitem{km-85}
{\sc S.~Kawashima and A.~Matsumura}, {\em Asymptotic stability of traveling
  wave solutions of systems for one-dimensional gas motion}, Comm. Math. Phys.,
  101 (1985), pp.~97--127.

\bibitem{knz03}
{\sc S.~Kawashima, S.~Nishibata, and P.~Zhu}, {\em Asymptotic stability of the
  stationary solution to the compressible {N}avier-{S}tokes equations in the
  half space}, Comm. Math. Phys., 240 (2003), pp.~483--500.

\bibitem{lmn98}
{\sc T.-P. Liu, A.~Matsumura, and K.~Nishihara}, {\em Behaviors of solutions
  for the {B}urgers equation with boundary corresponding to rarefaction waves},
  SIAM J. Math. Anal., 29 (1998), pp.~293--308.

\bibitem{matu-01}
{\sc A.~Matsumura}, {\em Inflow and outflow problems in the half space for a
  one-dimensional isentropic model system of compressible viscous gas}, Methods
  Appl. Anal., 8 (2001), pp.~645--666.
\newblock IMS Conference on Differential Equations from Mechanics (Hong Kong,
  1999).

\bibitem{mn-94}
{\sc A.~Matsumura and K.~Nishihara}, {\em Asymptotic stability of traveling
  waves for scalar viscous conservation laws with non-convex nonlinearity},
  Comm. Math. Phys., 165 (1994), pp.~83--96.

\bibitem{nn07-p}
{\sc T.~Nakamura and S.~Nishibata}, {\em Convergence rate toward planar
  stationary waves for compressible viscous fluid in multi-dimensional half
  space}, preprint 2008.

\bibitem{nny07}
{\sc T.~Nakamura, S.~Nishibata, and T.~Yuge}, {\em Convergence rate of
  solutions toward stationary solutions to the compressible {N}avier-{S}tokes
  equation in a half line}, J. Differential Equations, 241 (2007), pp.~94--111.

\bibitem{n85-bg}
{\sc K.~Nishihara}, {\em A note on the stability of travelling wave solutions
  of {B}urgers' equation}, Japan J. Appl. Math., 2 (1985), pp.~27--35.

\bibitem{ns-98}
{\sc M.~Nishikawa}, {\em Convergence rate to the traveling wave for viscous
  conservation laws}, Funkcial. Ekvac., 41 (1998), pp.~107--132.

\bibitem{sone-98}
{\sc Y.~Sone, F.~Golse, T.~Ohwada, and T.~Doi}, {\em Analytical study of
  transonic flows of a gas condensing onto its plane condensed phase on the
  basis of kinetic theory}, Eur. J. Mech, B/Fluids, 17 (1998), pp.~277--306.

\bibitem{unk-ns-pre}
{\sc Y.~Ueda, T.~Nakamura, and S.~Kawashima}, {\em Convergence rate toward
  degenerate stationary wave for compressible viscous gases}, preprint.

\bibitem{unk08}
{\sc Y.~Ueda, T.~Nakamura, and S.~Kawashima}, {\em Stability of
  degenerate stationary waves for viscous gases}, to appear in Archive for
  Rational Mechanics and Analysis.

\end{thebibliography}

} 

\end{document}